\documentclass[preprint] {imsart}
\RequirePackage[OT1] {fontenc}

\usepackage{mypackage}
\usepackage{mymacros2}
\usepackage{myenv}
\usepackage{mymacros_brackets}

\usepackage{DTR}

\definecolor{jcolor}{RGB}{041,122,000}
\definecolor{darkblue}{RGB}{000,000,150}
\definecolor{darkred}{RGB}{100,000,000}
\definecolor{purple}{RGB}{200,000,200}

\myexternaldocument{Supplement}
\begin{document}

\begin{frontmatter}
\title{Finding the Optimal Dynamic Treatment Regimes Using Smooth Fisher Consistent Surrogate Loss}

\author{\fnms{Nilanjana} \snm{Laha}\thanksref{b}\ead[label=e1]{nlaha@tamu.edu}},  
\author{\fnms{Aaron} \snm{Sonabend-W}\thanksref{a}\ead[label=e2]{asonabend@gmail.com}}, 
\author{\fnms{Rajarshi} \snm{Mukherjee$^\dag$}\thanksref{a}\ead[label=e3]{ram521@mail.harvard.edu}}
\and
\author{\fnms{Tianxi} \snm{Cai$^\dag$}\thanksref{a}\ead[label=e4]{tcai@hsph.harvard.edu}}
\address[b]{Department of Statistics, Texas A\&M,  College Station, TX 77843}
\address[a]{Department of Biostatistics, Harvard University, 677 Huntington Ave, Boston, MA 02115}
\footnotetext{$^\dag$: Equal Contributors}

\runtitle{Smooth Surrogates }

\runauthor{ Laha et al.}

\begin{abstract}
Large health care data repositories such as electronic health records (EHR) open new opportunities to derive individualized treatment strategies for complicated diseases such as sepsis. In this paper, we consider the problem of estimating sequential treatment rules tailored to a patient's individual characteristics, often referred to as dynamic treatment regimes (DTRs). Our main objective is to find the optimal DTR that maximizes a discontinuous value function through direct maximization of Fisher consistent surrogate loss functions. In this regard, we demonstrate that a large class of concave surrogates fails to be Fisher consistent -- a behavior that differs from the classical binary classification problems. We further characterize a non-concave family of Fisher consistent smooth surrogate functions, {\color{black} which\label{page: off-the-shelves} is amenable to gradient-descent type optimization algorithms}. Compared to the existing direct search approach under the support vector machine framework \citep{zhao2015}, our proposed DTR estimation via surrogate loss optimization (DTRESLO) method is more computationally scalable to large sample sizes and allows for broader functional classes for treatment policies. We establish theoretical properties for our proposed DTR estimator and obtain a sharp upper bound on the regret corresponding to our DTRESLO method. The finite sample performance of our proposed estimator is evaluated through extensive simulations. Finally, we illustrate the working principles and benefits of our method for estimating an optimal DTR for treating sepsis using EHR data from sepsis patients admitted to intensive care units.

\end{abstract}
\date{\today}

\end{frontmatter}

 \listoftodos[]

\section{Introduction}

Due to the increasing adoption of electronic health records (EHR) and 
the linkage of EHR with bio-repositories and other research
registries, integrated large datasets have become available for \textbf{real world evidence} based precision medicine studies. These rich EHR data capture heterogeneity in response to treatment over time and across patients, thereby offering unique opportunities to optimize treatment strategies for individual patients over time. Sequential treatment decisions tailored to patients' individual characteristics at given decision time points are often referred to as dynamic treatment regimes (DTRs) in the statistical literature and reinforcement learning (RL) in the machine learning literature. An optimal DTR can be defined as the sequential treatment assignment rule that maximizes the expected counterfactual outcome, often referred to as the value function in the DTR literature. 

To estimate the optimal DTR, the most traditional approaches rely on modelling the data-distribution  or part of the data-distribution \citep{xu2016, zajonc2012}. The most popular among the latter class are the regression-based methods, including Q-learning, A-learning and marginal structural mean models \citep[][]{Watkins1989, murphy2003, schulte2014, orellana2010, Robins2004}. The regression-based methods, especially Q-learning, offers the flexibility necessary for extension to a variety of settings including, but not limited to, semi-supervised setting \citep{sonabendw2021semisupervised}, interactive model-building \citep{laber2014}, discrete outcomes or utilities \citep{moodie2014} etc. However, 
the  underlying models in the regression-based approaches are often high-dimensional, and susceptible to mis-specification due to the sequential nature of the problem \citep{MurphySA2001MMMf}. Although A-learning and marginal structural mean models are more robust to model mis-specification, they still require the contrast of Q-functions to be correctly specified  \citep[cf.][]{schulte2014}. 
These limitations of the regression-based methods led the conception of the classification-based direct search methods, which, in contrast, directly targets  the counterfactual value function.

The classification-based approaches essentially rely
 on the representation of the counterfactual value function through importance sampling \citep{MurphySA2001MMMf}, whose maximization can be framed as a classification problem with respect to the zero-one loss function \citep[cf.][and the references therein]{zhao2012, zhao2015, chen2016personalized, zhou2017, song2015sparse, chen2017, cui2020semiparametric}. The resulting objective function is not amenable to efficient optimization  owing to the discontinuity of the zero-one loss. Therefore, following contemporary classification literature \citep[cf.][]{bartlett2006, lin2004note}  the  direct search methods aim to  replace the zero-one loss with alternative smoother fisher consistent surrogate loss functions to facilitate efficient classification methods.  The paradigm shift of estimating DTRs by finding classification rules is a powerful idea. Some authors indicate that existing  direct search methods  outperform regression-based counterparts  when the number of stages is small \citep{laber2019, luedtke2016}.

  Although initially developed for the one stage case, direct search method  was introduced to the  multi-stage DTR by the novel work of \cite{zhao2015}. 
 Currently, it has two mainstream approaches.  The first approach  performs binary classification stagewise   in a backward fashion  \citep[cf. BOWL method of ][]{zhao2015, jiang2019}. However, at stage $t$,  this approach can only use those observations whose treatment assignment matches the optimal treatment stage $t+1$ onward. As a result, the effective sample size of the initial stages dwindles rapidly, which can be problematic during practical implementation  \citep{laber2019, kallus2020comment}.  
 The other approach builds on a simultaneous optimization method, which utilizes the  whole data-set for estimating each treatment assignment \citep[simultaneous outcome weighted learning (SOWL),][]{zhao2015}. While it does not share the limitation of the BOWL-type approaches, this approach hinges on a sequential weighted classification problem which is complicated by the dependent nature of the DTR setting.  \cite{zhao2015}  solves this classification using a bivariate hinge-loss type surrogate. Although the idea behind simultaneous optimization is powerful,  the implementation via non-smooth hinge-loss surrogate leads to   a number of issues, scalability being one of them. See Section~\ref{sec: related literature} for more details.  
 It is natural to ask whether the hinge loss can be replaced by other surrogates. However, the answer is not immediate because unlike BOWL, the simultaneous classification does not yield to the  
binary classification theory on surrogate losses  \citep{bartlett2006}. Although multicategory and multi-label classifications have apparent resemblance with this  classification problem, as we will see, they have fundamental differences. This gives rise to the need for a unified study of fisher consistent surrogate losses under the DTR setting. Our paper is the first step towards that end.

For the ease of presentation, we focus on $k=2$ stage DTRs associated with two time points in this paper. However, the main  methodology easily extends to general k-stage settings when $k> 2$. Similar to most current works in direct search methods, we  consider only a binary treatment indicator, which is an important practical case \citep{laber2017}. Direct search with multi-level treatments would require substantially different techniques, and is out of the scope of the present paper. 

\subsection{Main contributions:}
In the sequel, we will refer to the classification problem resulting from the simultaneous optimization approach as ``the DTR classification problem" for brevity. We will refer to our approach of achieving optimal {\em DTR estimation via surrogate loss optimization} as DTRESLO. 
  \paragraph{Concave losses:}{
 In Theorem~\ref{thm: concave: main}, we establish that the above-bounded smooth concave surrogates fail to be Fisher consistent in the DTR context. The failure is not restricted to only smooth concave surrogates since our Theorem~\ref{theorem: hinge} also shows that  non-smooth hinge loss also fails to be Fisher consistent. 
    Furthermore, we have not encountered any concave loss function that is Fisher-consistent in the DTR context. Consequently, our findings naturally prompt the question of whether any concave loss function can indeed achieve Fisher consistency for this problem.
  }
  
  \paragraph{A class of Fisher consistent surrogates for DTR Estimation:}{
  Given the limited promise of concave surrogate losses for this problem, we directed our attention toward the realm of non-concave surrogates. We introduce a class of non-concave Fisher consistent surrogate losses (see Theorem~\ref{lemma: approximation error}), which are amenable to efficient gradient-based algorithms, such as stochastic gradient descent. This facilitates the utilization of fast and scalable optimization methods. {\color{black} {Since}\label{non-concave 1} the resulting optimizing problem is non-concave, convergence to the global maximum is not automatically guaranteed. However, the class of surrogate losses we consider do exhibit reliable empirical performance across all our simulation settings. } 
  Our approach offers flexibility for learning the DTRs so that practitioners can tailor the method to the data and problem at hand. 
 In particular, the smoothness of our surrogate losses makes the optimization problem suitable to a broad range of standard machine learning algorithms including, but not limited to, neural networks, wavelet series, and basis expansion. Interpretable treatment rules are also achievable by coupling our DTRESLO method with interpretable classifiers, such as linear or tree-based classifiers.
 Finally, since we optimize the primal objective function, variable selection in our case is straight-forward via addition of an $l_1$ penalty. 
  }
 
 \paragraph{Theoretical guarantee for a class of DTR estimators} We provide sharp upper bound on the regret -- the difference between the optimal value function and the value attained by  the estimated treatment regime, with detailed analyses focused on searching for DTR within the neural network classifiers. We perform a sharp analysis of our approximation error (see Theorem~\ref{thm: approximation error}) and  estimation  error  under Tsybakov's small noise condition  \citep{tsybakov2004}.
 In Corollaries~\ref{theorem: GE: neural network} and \ref{cor: wavelets: weak convergence}, we prove that provided the optimization error is small, the regret of our DTRESLO method with neural network and wavelet series classifiers decays at a fast rate. Here by \textit{fast}, we mean  decay rate faster than $n^{-1/2}$ is achievable. It turns out that this rate also matches 
 the minimax rate of risk decay (up to a poly-logarithmic factor) of binary classification  under  assumptions similar to ours \citep{audibert2007}.
  Since two stage DTR is unlikely to be simpler than one stage DTR, we conjecture that that our rate is   minimax-optimal (up to a poly-logarithmic factor) in two stage DTR under our assumptions. 
In the special case when treatment effects are bounded away from zero, we show that our regret decays at the rate of $O(1/n)$ up to a poly-logarithmic order. 

The rest of the article is organized as follows. In Section \ref{sec: DTR set-up} we outline the problem and discuss the mathematical formulation. In Section \ref{sec: concave} we discuss Fisher consistency in the DTR setting,  show that a large class of concave surrogates fail to be Fisher consistent, and  establish the Fisher consistency of a family of non-concave surrogates. In Section~\ref{sec: main methodology} we construct a method for estimating the optimal DTRs using the Fisher consistent surrogates, and discuss  the potential sources of error that contribute to the regret.  Section~\ref{sec: Approximation error} and Section~\ref{sec: generalization} are devoted towards obtaining theoretical upper bounds of the regret of our DTRESLO method.  Section~\ref{sec: Approximation error} focuses on approximation error, which is combined with the estimation error in Section~\ref{sec: generalization} to yield  the final regret bound. Section \ref{sec: opt: summary} provides a summary of the primary results concerning optimization error, with a comprehensive analysis available in Supplement \ref{sec:opt error}.
Then in section \ref{section: simulation & application} we illustrate our DTRESLO method's empirical performance with extensive simulations and an application to a sepsis cohort. We continue with a discussion in Section \ref{sec: discussion}. Additional details and proofs of our theoretical results are deferred to the Supplement.

\subsection{Notation}
\label{sec: notation}
  We let $\bRR$ denote the extended real line $\RR\cup\{\pm \infty\}$ and write $\RR_+$ for the positive half line $\{x\in\RR:x>0\}$. Denote by $\NN=\{1,2,\ldots\}$ the set of all natural numbers and for any integer $t$, we let $[t]=\{1,2,...,t\}$. We also let $\mathbb Z$ denote the set of all integers. For $m\in\NN$, we let $\|\cdot\|_m$ denote the $l_m$ norm , i.e. for $v\in\RR^m$, $\|v\|_m=(\sum_{i=1}^m |v_i|^m)^{1/m}$.
  If $v\in\NN^m$, we denote by $|v|_1$ the quantity $\sum_{i=1}^mv_i$. We let $B_{m}(0,K)$ denote the $l_2$-ball in $\RR^{m}$ centered at the origin with radius $K>0$. For two vectors $x,y\in\RR^{k}$, we let $\angle(x,y)$ denote the angle between $x$ and $y$.

  For any  probability measure $P$ and measurable function $f$, we denote by  $\|f\|_{P,k}$ the norm $\left(\int|f(x)|^k dP(x)\right)^{1/k}$. We will  also denote this norm by $L_k(P)$. Also,  $Pf$ will denote the integral $\int f dP$. For a concave function $f:\RR^k\mapsto\RR$, the domain  $\dom(f)$ will be defined as in \citep[][p. 74]{hiriart2004}, that is, $\dom(f)=\{x\in\RR^k\ :\ f(x)>-\infty\}$. 
For $f:\RR^2\mapsto \RR$, we denote by $f_{12}$ the partial derivative
\[f_{12}(x,y)=\pdv{f(x,y)}{x}{y}.\]
For any differentiable function $f:\RR^k\mapsto\RR$, $\grad f$ will denote the gradient of $f$, and the superlevel set of $f$ at level $c$ will be defined by  $\{x\in\RR^k: f(x)\geq c\}$.
For any $x\in\RR$, we denote by $\sigma(x)$ the ReLU activation function $x_+=\max(x,0)$.
 For any set $A$, use the notation $1[x\in A]$ to denote the event $\{x\in A\}$. Also, we denote by $\iint(A)$ the interior of the set $A$. The cardinality of $A$ will be denoted by $|A|$. Throughout this paper, we use the convention $\pm\infty\times 0=0$.
  In this paper, we will use $C$ and $c$ to denote generic constants which may vary from line to line. 
 
  Many  results in this paper are  asymptotic (in $n$) in nature and thus require some standard asymptotic notations.  If $a_n$ and $b_n$ are two sequences of real numbers then $a_n \gg b_n$ (and $a_n \ll b_n$) implies that ${a_n}/{b_n} \rightarrow \infty$ (and ${a_n}/{b_n} \rightarrow 0$) as $n \rightarrow \infty$, respectively. Similarly $a_n \gtrsim b_n$ (and $a_n \lesssim b_n$) implies that $\liminf_{n \rightarrow \infty} {{a_n}/{b_n}} = C$ for some $C \in (0,\infty]$ (and $\limsup_{n \rightarrow \infty} {{a_n}/{b_n}} =C$ for some $C \in [0,\infty)$). Alternatively, $a_n = o(b_n)$ will also imply $a_n \ll b_n$ and $a_n=O(b_n)$ will imply that $\limsup_{n \rightarrow \infty} \ a_n / b_n = C$ for some $C \in [0,\infty)$).

\section{Mathematical formalism}
\label{sec: DTR set-up}
We focus on the DTR estimation under a longitudinal setting where data are collected over time periods indexed by $t\in\{1,2\}$. Let $O_t\in\O_t \subset\RR^{p_t}$ denote the $p_t$ dimensional vector of patient clinical variables collected at time $t$ and $p=\max(p_1,p_2)$. At a given time $t$, a binary treatment decision $A_t\in\{\pm 1\}$ is made for the patient and a response to such treatment $Y_{t}\in\RR$ is observed.   Without loss of generality, we assume higher values of response $Y_{t}$ are desirable.  Let us denote the distribution underlying the observed random vector $\D = (O_1,A_1,Y_1,O_2,A_2,Y_2)$ by $\PP$. Suppose we sample $n$ i.i.d. observations from $\PP$. The corresponding empirical distribution function will be denoted by $\PP_n$.  Since treatment decisions are often made based on all previous states including prior treatments and responses, we define the patient history by
\[H_1=O_1,\mbox{ and }H_2=(O_1, Y_1, O_2, A_1),\]
where $H_1$ and $H_2$ take values in sets $\H_1$ and $\H_2$, respectively. 
We denote by $\pi_1(a_1\mid H_1)$ and $\pi_2(a_2\mid H_2)$  the propensity scores $\PP(A_1=a_1|H_1)$ and $\PP(A_2=a_2|H_2)$, respectively. 


Our goal is to find the treatment regime $d=(d_1,d_2): \H_1 \times \H_2 \mapsto \{\pm 1\}\times \{\pm 1\}$ that maximizes the expected sum of rewards $Y_1(d)+Y_2(d)$, $$V(d_1,d_2)=\E_d [Y_1(d)+Y_2(d)],$$ 
where $Y_t(d)$ is the potential outcome associated with time $t \in \{1,2\}$ , and $\E_d$ is the expectation with respect to the data distribution under regime $d$.  To this end, first we make some assumptions on the observed data distribution $\PP$ so that $V(d_1,d_2)$ becomes identifiable under $\PP$. 

\noindent \textbf{Assumptions for identifiability:}
\begin{itemize}
    \item[I.] \textbf{Positivity:} There exists a constant $C_{\pi}\in(0,1)$ so that $\pi_t(A_t \mid H_t)>C_\pi$ for all $H_t$, $t=1,2$.
    \item[II.]\textbf{Consistency:} The observed outcomes $Y_t$ and covariates $O_t$ agree with the potential outcomes and covariates  under the treatments actually received. See \cite{schulte2014, robins1994correcting} for more details.
    \item[III]\textbf{Sequential ignorability:} For each $t=1,2$, the treatment assignment $A_t$ is conditionally independent of the future potential outcomes $Y_{t}$ and future potential clinical profile $O_{t+1}$ given $H_t$. Here we take $O_3$ to be the empty set.
\end{itemize}
Our version of sequential ignorability follows from  \cite{robins1997causal, MurphySA2001MMMf}. Assumptions I-III are standard in DTR literature  \citep[e.g.]{schulte2014,sonabendw2021semisupervised, MurphySA2001MMMf, zhao2015}.  

Under Assumptions I-III,   $\E_d (Y_1+Y_2)$ can be identified under $\PP$ as \citep{zhao2015} 
\begin{align*}
\E_{d} [Y_1(d)+Y_2(d)]=\PP\lbt \dfrac{( Y_1+Y_2) 1[A_1=d_1(H_1)] 1[A_2=d_2(H_2)]}{ \pi_1(A_1\mid H_1)\pi_2(A_2\mid H_2)}  \rbt.
\end{align*}
The treatment effect contrasts are defined as follows:
\begin{equation}\label{def: treatment effect stage 1}
 \mathcal T_1(H_1)=    \E[Y_1+U_2^*(H_2)|A_1=1, H_1]-\E[Y_1+U_2^*(H_2)|A_1=-1, H_1],
 \end{equation}
 and
\begin{equation}\label{def: treatment effect stage 2}
 \mathcal T_2(H_2)=   \E[Y_1+Y_2|A_2=1,H_2]-\E[Y_1+Y_2|A_2=-1,H_2],
 \end{equation}
 where
 \begin{equation}\label{def: U 2 star}
    U_2^*(H_2)=\max_{a_2\in\{\pm 1\}}\E[Y_2| H_2,A_2=a_2].
\end{equation}
The above quantities are also called the optimal-blip to-zero function, or sometimes simply the blip function, in the literature \citep{Robins2004, schulte2014, luedtke2016}. We will also refer to them as the first  stage and the second stage conditional treatment effects. 
For the blip functions or the conditional treatment effects to be well defined, we need the conditional expectations in \eqref{def: treatment effect stage 1} and \eqref{def: treatment effect stage 2} to be finite, which is not automatically guaranteed by Assumptions I-III. Therefore we introduce another assumption to ensure that the treatment effects are well-defined.
\begin{itemize}
    \item \textbf{Assumption IV.}  For any $h_2\in\H_2$, and $a_2\in\{-1,1\}$, the conditional expectation $\E[|Y_1|+|Y_2|\mid H_2=h_2, A_2=a_2]<\infty$. For any $h_1\in\H_1$, and $a_1\in\{-1,1\}$, the conditional expectation $\E[Y_1+U_2^*(H_2)\mid H_1=h_1, A_1=a_1]<\infty$. Furthermore, $\E[|Y_1+Y_2|]<\infty$.
\end{itemize}
In addition to ensuring the well-definedness of treatment effects, Assumption IV also serves as a  technical requirement in our proofs and enhances the interpretability of our theoretical findings. While we expect that many of our theoretical results would hold even without this assumption, the proofs would become more intricate and cumbersome. It is important to note that Assumption IV is not overly stringent since, in most of our applications, $Y_1$ and $Y_2$ represent measurements and are automatically bounded.

{\color{black} We\label{page: d star def} define the optimal DTR $d^*$ to be the maximizer of $\E_d [Y_1(d)+Y_2(d)]$ over all possible regimes $d=(d_1,d_2)$ such that $d_1:\H_1\mapsto\{\pm 1\}$ and $d_2:\H_2\mapsto \{\pm 1\}$. } Under Assumptions I-III, the optimal policy $d^*$ can be identified as follows \citep{zhao2015,chakraborty2013}
\begin{align}\label{def: d1 and d2 star}
d_2^*(H_2)&=\argmax_{a_2\in\{\pm 1\}}\E[Y_2\mid  H_2,A_2=a_2]\nn\\
d_1^*(H_1)&=\argmax_{a_1\in\{\pm 1\}}\E\left[ Y_1+ U_2^*(H_2) \mid H_1, A_1=a_1\right],
\end{align}
 where $U_2^*$ is as defined in \eqref{def: U 2 star}. 
Since the optimal decision rules
remain unchanged if a constant $C$ is added to both $Y_1$ and $Y_2$, in what follows, unless otherwise mentioned, we assume that $Y_1,Y_2>C$ for some $C>0$. 
This trick was also used in \cite{zhao2015}.

\begin{remark}[Uniqueness of $d_1^*$ and $d_2^*$]
\label{remark: uniqueness of d1star and d2star}
It is worth noting that $d_1^*$ and $d_2^*$ defined in \eqref{def: d1 and d2 star} may not be unique because they are allowed to take any value in $\{\pm 1\}$ at the boundary. To elaborate on this further, suppose some $H_2$ satisfies
\[\E[Y_2| H_2,A_2=1]=\E[Y_2| H_2,A_2=-1].\]
Such values of $H_2$ constitute the decision boundary for the second stage.
Then both versions $d_2(H_2)=1$ and $d'_2(H_2)=-1$ qualify as  optimal rule for at  $H_2$. Similarly, for $d_1^*$, we can show that if $H_1$ belongs to the first stage decision boundary
\[\lbs h_1\in\mathcal H_1\ :\ \E\slbt Y_1+U_2^*( H_2)\bl A_1=1,h_1\srbt=\E\slbt Y_1+U_2^*( H_2)\bl A_1=-1,h_1\srbt\rbs,\]
then $d_1^*(H_1)$ can take either value $+1$ or $-1$. Thus, $d_1^*$ is not unique either.
Consequently, to avoid confusion, we let $d_1^*=1$ and $d_2^*=1$ at both first and second stage decision boundaries. Note that under this convention, $d_1^*(H_1)=1[T_1(H_1)\geq 0]$ and $d_2^*(H_2)=1[\mathcal T_2(H_2)\geq 0]$. In what follows, we  shall also refer to this optimal rule as ``the optimal rule".
\hfill\qedsymbol{}
\end{remark}

There is an alternative way of formulating $d^*$. 
If $(f_1^*,f_2^*)$
is a maximizer of
\begin{equation}\label{def: value function}
   V(f_1,f_2)=\PP\lbt \dfrac{( Y_1+Y_2 ) 1[A_1f_1(H_1)>0]1[A_2f_2(H_2)>0]}{ \pi_1(A_1\mid H_1)\pi_2(A_2\mid H_2)}  \rbt 
\end{equation}
over the class
\begin{equation}\label{def: class F}
\mathcal F=\lbs(f_1,f_2) \ \bl \ f_1:\mathcal{H}_1\mapsto\RR,\quad f_2:\mathcal{H}_2\mapsto\RR \text{ are measurable}\rbs,
\end{equation}
then $\sign(f_1^*)$ and $\sign(f_2^*)$ yield the optimal rules $d_1^*$ and $d_2^*$, respectively \citep{zhao2015}. If $f_1^*$ and $f_2^*$ take the value zero, then $d_1^*$ and $d_2^*$ can be either $+1$ or $-1$. Finally, even if $d_1^*$ and $d_2^*$ are unique,  $f_1^*$ and $f_2^*$ need not be unique.

At this stage, although it is intuitive to consider maximization of the sample analogue of $V(f_1,f_2)$ to estimate the optimal decision rule,  the non-concavity and discontinuity of the  zero-one loss function render the maximization of $V(f_1,f_2)$  computationally hard.  To deal with issues of similar flavor, the  classification literature \citep[cf.][]{bartlett2006} suggests using a  suitable surrogate to the zero-one loss function. We appeal to this very intuition and consider 
\begin{equation}\label{def: objective: ours}
V_{\psi}(f_1,f_2)=\PP \lbt\dfrac{(Y_1+Y_2)\psi\slb A_1 f_1(H_1), A_2f_2(H_2)\srb}{\pi_1(A_1\mid H_1)\pi_2(A_2\mid H_2)} \rbt,
\end{equation}
where $\psi$ is some bivariate function. For example, \cite{zhao2015} takes $\psi(x,y)=\min(x-1,y-1,0)$, the bivariate concave version of the popular hinge loss $\phi(x)=\max(1-x,0)$.

 Suppose there exist functions $f_1:\H_1\mapsto[-\infty,\infty]$ and $f_2:\H_2\mapsto[-\infty,\infty]$ so that 
 \begin{equation}\label{def: tilde f}
V_{\psi}(\tilde f_1,\tilde f_2)=\sup_{(f_1,f_2)\in\mathcal F}V_{\psi}(f_1,f_2)
\end{equation}
 where $\mathcal F$ is as defined in \eqref{def: class F}. Note that  $\tilde f_1$ and $\tilde f_2$  may not be  unique. 
Each $(\tilde f_1,\tilde f_2)$ lead to the decision rules $\tilde d_1(H_1)=\sign(\tilde f_1(H_1))$ and
$\tilde d_2(H_2)=\sign(\tilde f_2(H_2))$.  If $\tilde f(H_t)=0$, then $\tilde d_t(H_t)$ can be either $+1$ or $-1$. 
    We let $\tilde f_1$ and $\tilde f_2$  to be extended-valued functions because the supremum on the right hand side of \eqref{def: tilde f} may not be attained in $\F$ for some surrogates. It may happen that the supremum of $V_\psi$ over $\mathcal F$ is attained at some $ f_1$ and $ f_2$ which satisfies $f_1(H_1)=\infty$ or $-\infty$  (alternatively,   $ f_2(H_2)=\infty$ or $-\infty$).  Although $\tilde f_t$ can be extended valued, it does not create much technical issues because (a) $\tilde d_t$ is always  $1$ or $-1$ for $t=1,2$, and $\tilde d_t$ is the  object of interest here.

Finally, we define excess risk in line with the excess risk in context of classification. 
Letting $V^{*}=V(f_1^*,f_2^*)$
and $V^{*}_\psi=V_{\psi}(\tilde f_1,\tilde f_2)$,  
we define the respective regret and  $\psi$-regret of using $(f_1,f_2)$ by
\[V^{*}-V(f_1,f_2)\quad\text{and}\quad V^{*}_{\psi}-V_{\psi}(f_1,f_2),\]
respectively.
 Note that regret and the  $\psi$-regret are always non-negative. 
 
Throughout our paper, we will compare our DTR  classification with binary classification. Therefore, we  will fix the notation for  binary classification.
In the setting of binary classification, we have observations $X$ taking value in an Euclidean space $\mathcal X$. Each  $X$ is associated with  a label $A$, which plays the same role as our treatment assignments. The optimal rule or the  Bayes rule  assigns label $+1$ if $\eta(X)=P(A=1\mid X)>1/2$ and label $-1$ otherwise \citep[cf.][]{bartlett2006}. If $\eta(X)=1/2$, both labels are optimal.  The Bayes rule minimizes the classification risk $\R(f)=\P(Af(X)<0)$ over all measurable functions $f:\mathcal X\to\RR$. Also, we denote $\mathcal R^*=\inf_f \mathcal R(f)$, where the infimum is taken over all measurable functions.
Replacing the zero-one loss with the  surrogate $\phi$  results in the  $\phi$-risk
$\mathcal R_{\phi}(f)=\E[ \psi(Af(X))]$. We let $\R_{\phi}^{*}$ denote the optimized $\phi$-risk.


 Some parallels with the DTR classification setting are immediate. For example,   $V(f_1, f_2)$, $V^*$, $V_{\psi}(f_1, f_2)$, and $V^*_{\psi}$ correspond to $R(f)$, $R^*$, $R_{\phi}(f)$,  and $R^*_{\phi}$,  respectively. next, defining the maps $\G:\mathcal H_1\mapsto \RR$ and {\color{black} $\Gt:\mathcal H_2\mapsto\RR$ } by
  \begin{equation}\label{def: G}
     \G(H_1)=\frac{\E[Y_1+U_2^*(H_2)|A_1=1, H_1]}{\E[Y_1+U_2^*(H_2)|A_1=1, H_1]+\E[Y_1+U_2^*(H_2)|A_1=-1, H_1]},
 \end{equation}
\begin{equation}\label{def: Gt}
    \Gt(H_2)=\frac{\E[Y_1+Y_2|A_2=1,H_2]}{\E[Y_1+Y_2|A_2=1,H_2] +\E[Y_1+Y_2| A_2=-1,H_2]},
 \end{equation}
 we observe that
 $\G$ and $\Gt$ play the same role in DTR setting as the conditional probability  $\eta$ in context of binary  classification. To elaborate, from the definitions of $d_1^*$ and $d_2^*$ in \eqref{def: d1 and d2 star}, it follows that $d_t^*(H_t)=+1$ if $\eta_{t}(H_t)>1/2$, and $-1$ otherwise.
 Note also that  the first stage and second stage decision boundaries can be represented by the sets $\{h_1:\G(h_1)=1/2\}$ and $\{h_2:\Gt(h_2)=1/2\}$.

Throughout this paper, we occasionally make statements such as $\G(H_1)\geq 1/2$, $T(H_2,A_2)\geq 0$, $d_1^*(H_1)\neq \tilde d_1(H_1)$, etc. Since $H_1$, $H_2$, $A_1$, $A_2$, etc. are random variables, quantities like $\G(H_1)$, $\Gt(H_2)$, $T(H_2,A_2)$, $d_2^*(H_1)$, $d_1^*(H_1)$ are also random. To avoid any confusion, we wish to clarify that when such statements are made, it implies that the stated conditions hold for all realizations of $H_1$, $H_2$, $A_1$, $A_2$, etc.


 \section{Fisher consistency}
 \label{sec: concave}
A desirable  $\psi$ should ensure that $\tilde d$ is consistent with $d^*$. To concertize the idea, we need the concept of Fisher consistency.
\begin{definition}
\label{def: Fisher consistency}
The surrogate $\psi$ is called Fisher consistent if for all $\PP$ satisfying Assumption I-IV, any $\{f_{1n},f_{2n}\}_{n\geq 1}\subset \F$ that satisfies 
\[V_\psi(f_{1n},f_{2n})\to V_\psi^*,\quad\text{also satisfies}\quad V(f_{1n},f_{2n})\to V^*.\]
\end{definition}
 Our definition of Fisher consistency is in line with classification literature \citep{bartlett2006}. Note that Definition~\ref{def: Fisher consistency} does not require $\tilde f_1$ and $\tilde f_2$ to exist or be measurable. However, if $\tilde f_1$ and $\tilde f_2$ do exist, and they are in $\F$, {\color{black} then\label{page: FC def} Fisher consistency implies $V(\tilde d)=V^*$, indicating $\tilde d$ is a candidate for $d^*$.}
In context of binary classification, the surrogate $\phi$ is Fisher consistent if and only if $\mathcal R_\phi(f_n)\to \mathcal R_\phi^*$ implies $\mathcal R(f_n)\to \R^*$, where  $f_n$'s are measurable functions mapping $\mathcal X$ to $\RR$.

\begin{remark}[Characterization of Fisher consistency]
 In many classification problems, e.g. binary, multicategory, or multi-label classification, Fisher consistency can be directly characterized by  convex hulls of points in the image space of $\psi$, and the related notion is known as calibration \citep{bartlett2006, zhang2010, tewari2007, gao2011}. For example, Theorem 1 of \cite{bartlett2006} shows that a surrogate $\phi$ is Fisher-consistent for binary classification if and only if the following condition holds.
 \begin{condition}\label{cond: binary classification}
 $\phi:\RR\mapsto\RR$ satisfies
 \[ \sup_{x: x(2\eta-1)\leq 0}\slb \eta\phi(x)+(1-\eta)\phi(-x)\srb< \sup_{x\in\RR}\slb \eta\phi(x)+(1-\eta)\phi(-x)\srb\]
 for all $\eta\in[0,1]$ such that $\eta\neq 1/2$.
 \end{condition}
 However, due to the sequential nature of the DTR set-up,  it is not easy to represent Fisher consistency in terms of analytical properties of $\psi$. This complicates the analysis of Fisher consistency in the DTR set-up.
 \hfill\qedsymbol
\end{remark}

Traditionally, the first preference  of surrogate losses have been the concave (convex in context of minimization) surrogates because they ensure unique optimum \citep{chen2017}. 
In the binary setting, a univariate concave surrogate $\phi$ is Fisher consistent if and only if it is differentiable at $0$ with positive derivative \cite[see][Theorem 6]{bartlett2006}.   Many commonly used univariate concave losses satisfy these conditions. We display some of these in Figure \ref{Figure: 1d functions}. 
An important geometric property of these functions is that they  mimic the graph of the zero-one loss function. After proper shifting and scaling, their image lies below that of the zero-one loss function (see Figure~\ref{Figure: 1d functions}). 
Of course, concavity is not necessary for classification-calibration, and this geometric property is shared by non-concave classification calibrated losses as well \citep[see Lemma 9 of][]{bartlett2006}.
\begin{figure}
\includegraphics[width=.5\textwidth]{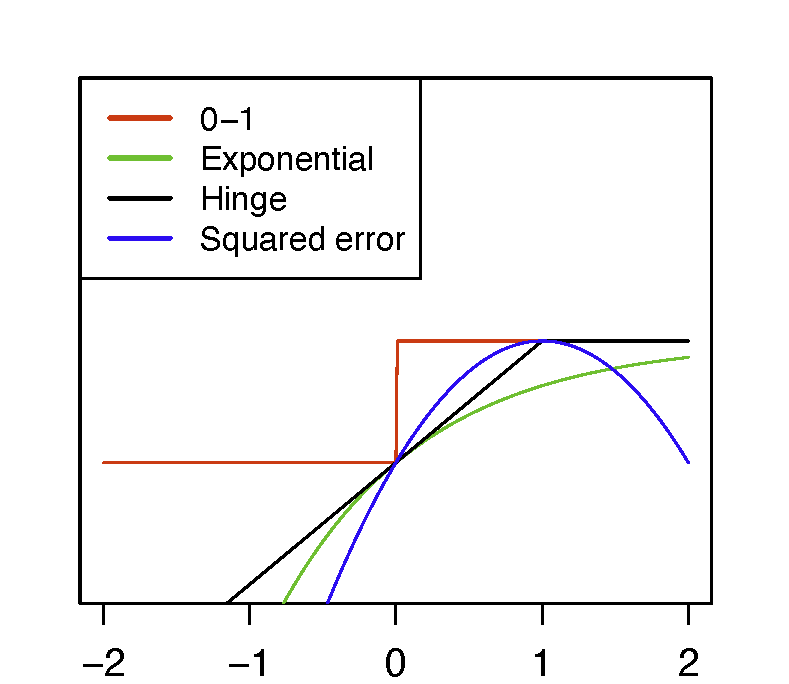}
\caption{Plots of $\phi(x)$ vs $x$ for concave calibrated value function $\phi$ for binary decision rules. Here are the functions, Zero-one: $\phi(x)=1[x>0]$,  Exponential: $\phi(x)=1-e^{-x}$, Hinge: $\phi(x)=\min(x,1)$, Squared error: $\phi(x)=1-(1-x)^2$.}
\label{Figure: 1d functions}
\end{figure}

There are also classes of concave surrogates which are Fisher consistent for multicategory classification with respect to the zero-one loss \citep{duchi2018, tewari2007, neykov2016} or for multilabel classification with respect to Hamming loss \citep{gao2011}. In that light, it is not unnatural to expect concave surrogates will succeed in the DTR classification setting as well. Unfortunately, as we will see in the next section, this simple-minded extension of binary classification may not hold.

 DTR classification bears resemblance with  multilabel classification  \citep{dembczynski2012}  but additional complication arises since $H_2$ contains $H_1$. Also, the Fisher consistency literature on multilabel classification \citep{gao2011} is based on Hamming loss and partial ranking loss, which are substantially different from the zero-one loss.  Our problem also exhibits similarity with multiclass classification \citep{duchi2018}. However, a big difference arises because of the sequential structure. Had $d_1$ been a map from $\H_2$ to $\{\pm 1\}$ similar to $d_2$, 
existing theory on multiclass classification  \citep{duchi2018} could be readily used to provide conditions for a general function $\psi$ to be Fisher consistent. However, during the treatment assignment $d_1$, one has no knowledge of  $A_1$, $Y_1$, and $O_2$. \cite{tewari2007} and \cite{zhang2010} develop tools for  general classification set-ups, but these  tools are too generalized for explosion of the specific sequential structure of DTR classification. In fact, it is the binary classification, which seems to have the most parallels with DTR classification.

 \begin{figure}[h]
  \begin{subfigure}[b]{.3\textwidth}
     \includegraphics[width=\textwidth]{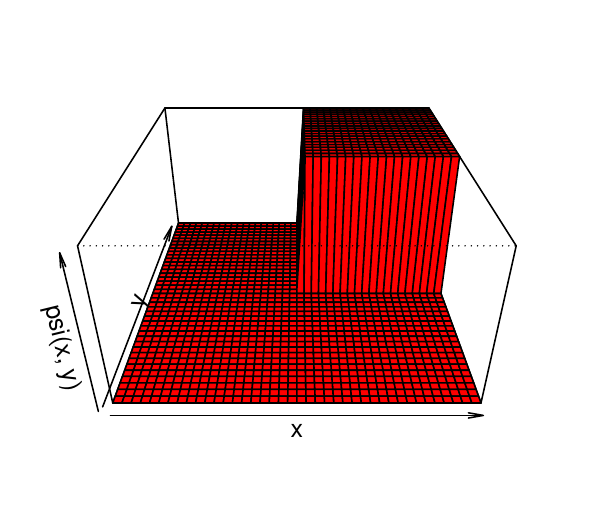}
 \caption{Bivariate zero-one loss:  $1[x>0, y>0]$.}
 \label{Fig: biv 0-1}
 \end{subfigure}~
 \begin{subfigure}[b]{.3\textwidth}
     \includegraphics[width=\textwidth]{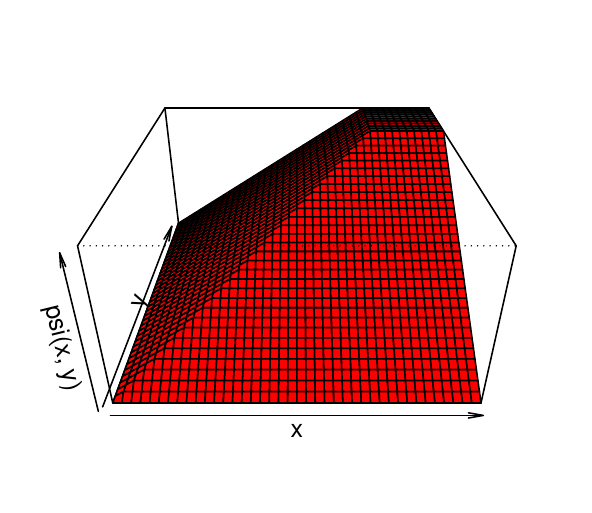}
 \caption{Bivariate hinge loss:   $\max(x-1,y-1,0)$.}
 \label{Fig: biv hinge}
 \end{subfigure}\\
 \begin{subfigure}[b]{.3\textwidth}
     \includegraphics[width=\textwidth]{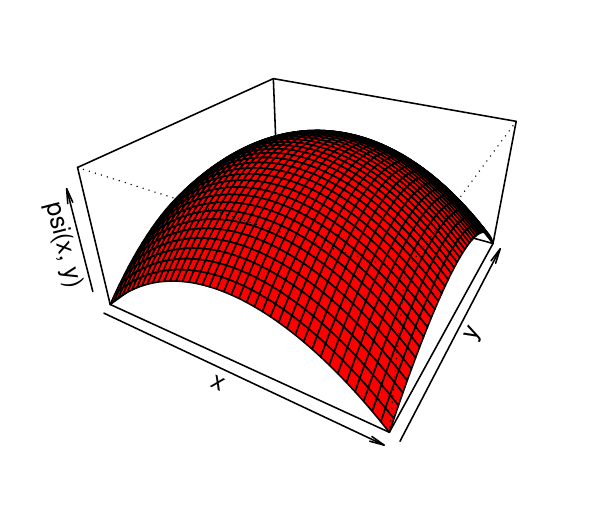}
 \caption{Bivariate squared error loss:  $-(x^2+y^2)$.}
 \label{Fig: sq error}
 \end{subfigure}~
 \begin{subfigure}[b]{.3\textwidth}
     \includegraphics[width=\textwidth]{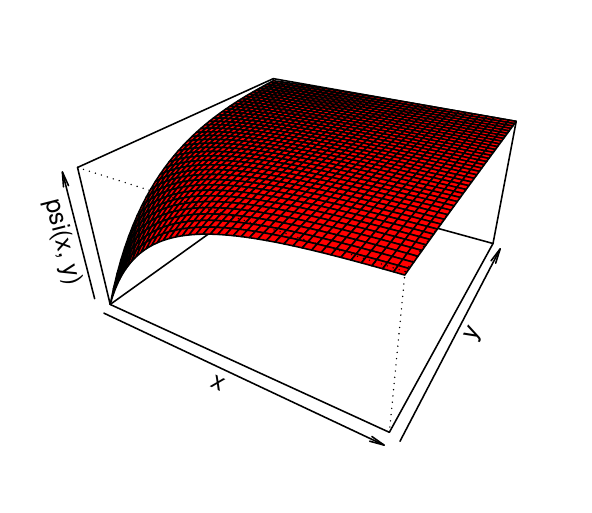}
 \caption{Bivariate exponential loss:  $-\exp(-(x+y))$.}
 \label{Fig: biv exp}
 \end{subfigure}
 ~
  \begin{subfigure}[b]{.3\textwidth}
     \includegraphics[width=\textwidth]{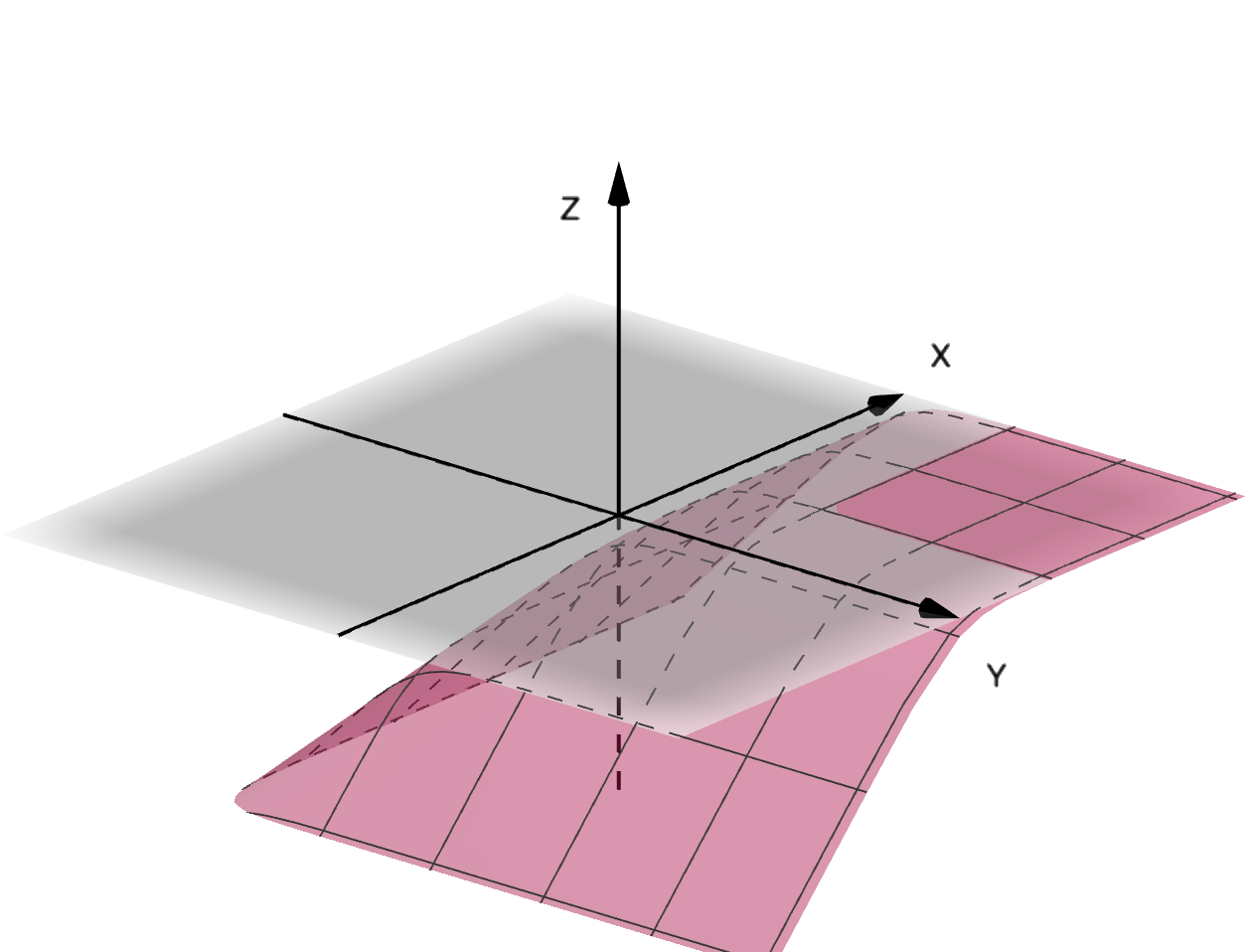}
 \caption{Bivariate logistic loss:  $-\log(1+\exp(-x)+\exp(-y))$.}
 \label{Fig: biv logistic}
 \end{subfigure}
 \caption{Bivariate zero-one loss function and some concave surrogate losses}\label{Fig: bivariate Loss functions}
 \end{figure}

 \subsection{Concave surrogates}
 \label{sec: failure of smooth concave}
In this section, we will establish that a large class of concave surrogates fail to be Fisher consistent for DTR estimation. 
We first consider the case of smooth concave losses because they amend to gradient based optimization methods with good scalability properties. 
We will start our discussion with an example. The smooth concave function $\phi(x)=-\exp(-x)$ is Fisher consistent in the binary classification setting. Let us consider its bivariate extension $\psi(x,y)=-\exp(-x-y)$. It tuns out that $\tilde d_1(H_1)$ takes the form 
\begin{align}
\label{eq: prelude to surrogate}  
\argmax_{a_1\in\{\pm 1\}}\E\slbt h\slb Y_1+\E[Y_2\mid H_2,A_2=1],Y_1+\E[Y_2\mid H_2,A_2=-1\srb\ \Big |\ H_1, A_1=a_1\srbt,
\end{align}
where $h(x,y)=\sqrt{xy}$. However, $d_1^*(H_1)$ takes the same form but with $h(x,y)=\max(x,y)$. In general, therefore, $\tilde d_1(H_1)$ and $d_1^*(H_1)$ do not agree. To see this, consider the toy example when $Y_1=1$ and 
\begin{align}
\label{def: easy example}
    Y_2=&\ 4\times 1[A_1,A_2=1]+3\times 1[A_1=1,A_2=-1]\nn\\
    &\ +5\times 1[A_1=-1,A_2=1]+1[A_1,A_2=-1].
\end{align}
In this case, $d_1^*(H_1)=-1$ but $\tilde d_1(H_1)=1$ for all $H_1$, and  clearly, $\psi$ is not Fisher consistent. If we consider other examples of smooth concave $\psi$, e.g. logistic or quadratic loss, we obtain different $h$, but for these examples as well, $h$ is  quite different from the non-smooth $h(x,y)=\max(x,y)$. 

The above heuristics indicate  that the criteria of DTR Fisher consistency may be incompatible with  smooth concave losses. 
Theorem~\ref{thm: concave: main} below concretize the above heuristics for an important class of concave smooth losses.
Theorem~\ref{thm: concave: main} assumes that $\psi$ is closed and strictly concave. We say a  function is closed if it is upper semicontinuous everywhere, or equivalently, if its superlevel sets are closed \citep[pp. 78,][]{hiriart2004}. The function  
$h$ is strictly concave if for any $\lambda\in(0,1)$, and $x,y\in\dom(h)$,
 \[h(\lambda x+(1-\lambda)y)>\lambda h(x)+(1-\lambda)h(y).\]
 \begin{theorem}\label{thm: concave: main}
  Suppose  $\psi$ is closed, strictly concave, and bounded above. In addition, $\psi$ has continuous second order partial derivatives and $\psi_{12}$ has continuous partial derivatives on $\iint(\dom(\psi))$. 
   Then $\psi$ can not be Fisher consistent for two stage DTR.
  \end{theorem}
 We list below some examples of $\psi$, also shown in Figure~\ref{Fig: bivariate Loss functions}, which satisfy the assumptions of Theorem~\ref{thm: concave: main}. 
 \begin{alignat*}{2}
     & \mbox{\bf Exponential:} & \quad & \psi(x,y)=-\exp(-x-y); \\
     & \mbox{\bf Logistic:} & \quad & \psi(x,y)=-\log(1+e^{-x}+e^{-y}); \\
     & \mbox{\bf Quadratic:} & \quad & \psi(x,y)=z^TQz+b^Tz+c, \mbox{where $z=(x,y)^T$,}\\
     & & & \mbox{ $Q$ is negative definite, $b\in\RR^2$, and $c\in\RR$.}
 \end{alignat*}
  The proof of Theorem~\ref{thm: concave: main} is given in Supplement~\ref{sec: proof: concave}.  Our counterexample for Theorem~\ref{thm: concave: main} is  based on a pathological case where $O_2$ and $Y_1$ are deterministic functions of $H_1$. We chose this case 
  because it grants technical simplification. The  realistic cases are no more likely to yield under concave surrogates than this simple pathological case. The calculations underlying the proof of Theorem~\ref{thm: concave: main} become severely technically challenging when the second stage covariates are potentially random given $H_1$.

\begin{remark}[Main challenges in the proof of Theorem~\ref{thm: concave: main}]
 A main difficulty in proving Theorem~\ref{thm: concave: main} is that even under our pathological case, $\tilde f_1(H_1)$ and $\tilde f_2(H_2)$ do not have closed form expressions. They are implicitly defined as maximizers of complex functionals of $\psi$. Therefore, if we consider a very large class of $\psi$'s,  characterization of $\tilde f_1(H_1)$ and $\tilde f_2(H_2)$ becomes difficult.
The assumptions on $\psi$ ensure that the class of $\psi$'s under consideration is manageable, mitigating some technical difficulties in the 
characterization of $\tilde f_1(H_1)$ and $\tilde f_2(H_2)$. The latter is essential for learning the behaviour of the signs of $\tilde f_1(H_1)$ and $\tilde f_2(H_2)$. Thus the assumptions  on $\psi$ are required for  technical reasons in the proof. That is to say that our conditions on $\psi$ are probably not necessary, and the assertions of Theorem~\ref{thm: concave: main} may hold even without these assumptions. In fact, we are not aware of any concave surrogates that are Fisher consistent in this context. We defer further discussion on the  assumptions in Theorem~\ref{thm: concave: main} to Supplement~\ref{sec: assumptions of concave}. \hfill\qedsymbol
  \end{remark}

{\color{black} \label{page: smoothness assumption}The smoothness assumption in Theorem~\ref{thm: concave: main} is a technical assumption. Specifically, the existence of a gradient of $\psi$ makes the proof simpler. However, we believe that the result may continue to hold without this condition, albeit with a more technically involved proof. In particular,  the negative result  in Theorem~\ref{thm: concave: main} is unlikely to be an artifact of the smoothness of $\psi$ in Theorem~\ref{thm: concave: main}, and may hold for broader classes of concave functions. In support of this claim, in Section  \ref{sec: hinge loss}, we demonstrate  that a concave variant of the bivariate hinge loss $\min(x-1, y-1, 0)$, a commonly used non-smooth concave loss, is not Fisher consistent. In fact, to our knowledge, there exists no concave surrogate, whether smooth or not, that is Fisher consistent for the DTR classification problem. These observations lead us to suspect that no concave loss is Fisher consistent for the DTR problem.} 

While we do not have an intuitive explanation for the apparent failure of  concave functions, 
 we attempt at making one heuristic reasoning. Even in the one stage case of binary classification, it was observed that  Fisher consistency requires the surrogates to mimic the shape of the zero-one loss to some extent.
It appears to us that for Fisher consistency in two stage DTR, the function $\psi$ has to mimic the shape of the bivariate zero-one loss function (see Figure~\ref{Fig: biv 0-1}) more closely than that was necessary in  binary classification (see Figure~\ref{Figure: 1d functions}). In other words, the non-concavity of the zero-one loss function at the origin  pushes the concave losses to failure, thereby necessitating search for $\psi$ among non-concave losses, which we will study in Section \ref{sec: fisher consistent surrogate}.

Smooth concave or convex surrogates  fail to be Fisher consistent in many other complex machine-learning problems.
 For example, \cite{gao2011} shows known convex surrogates are not Fisher consistent  for multilabel classification with ranking loss. 
Ranking is another notable example, where convex losses fail for a number of losses including the pairwise disjoint loss \citep{calauzenes2012}. In fact, in the latter case, the existence of a Fisher consistent concave surrogate  would imply that the feedback arc-set problem  ispolynomial-time solvable \citep{duchi2010}, which is conjectured to be NP complete \citep{karp1972}.  The DTR classification problem shares one common feature with 
  the above-stated  machine-learning problems where these surrogate losses fail. It does not \emph{organically} reduce to a sequence of weighted binary classification problems, which appears to be a common element  of all  classification problems that are solvable via convex surrogates, e.g. multicategory loss with zero-one loss function \citep{tewari2007}, multilabel classification with partial ranking and hamming loss \citep{gao2011}, ranking with Hamming loss \citep{calauzenes2012}, ordinal regression with absolute error loss \citep{pedregosa2017} etc. Here we emphasize the word ``organic" because DTR classification does reduce to sequences of binary classification if it is framed as a sequential classification via  exclusion of data points at each stage; cf. BOWL  \citep{zhao2015}.
 
 \subsection{Hinge loss}
 \label{sec: hinge loss}
{\color{black} In this section, we demonstrate the Fisher inconsistency of the non-smooth loss function $\psi(x,y)=\min(x,y,1)$, which is a bivariate version of the univariate hinge loss $\min(x,1)$.   The Fisher inconsistency of the hinge loss provides support to the conjecture that the Fisher inconsistency of concave surrogates extends beyond the class of smooth losses. The specific form of the hinge loss we examine has also been explored by \cite{zhao2015} as well. See Figure \ref{Fig: biv hinge} for a pictorial representation of this loss.   If desired, readers may choose to bypass this section and proceed directly to Section \ref{sec: fisher consistent surrogate}, which focuses on the study of Fisher-consistent losses.

 \cite{zhao2015} suggested a location transformation of the outcomes  $Y_1$ and $Y_2$ so that they become positive, which is in alignment with our discussion in Section \ref{sec: DTR set-up}. Since we mainly focus on their implementation of the hinge loss, we will take $Y_1$ and $Y_2$ to be positive for the time being.
     For our hinge loss, it turns out that we can especially characterize the solution $\tilde d$. The following inequality will be crucial for understanding the form of $\tilde d$ in this case: 
\begin{align}
 \label{inlemma: hinge solution first stage requirement}
 & \abs{\E[T(H_2,d_2^*(H_2))\mid H_1=h_1,A_1=1]-\E[T(H_2,d_2^*(H_2))\mid H_1=h_1,A_1=-1] }\nn\\
> &\ \E[T(H_2,-d_2^*(H_2))\mid H_1=h_1,A_1=1] +\E[T(H_2,-d_2^*(H_2))\mid H_1=h_1,A_1=-1],
\end{align}
where we remind the readers that $T(H_2, a_2)=Y_1+ \E[Y_2\mid H_2, A_2=a_2]$. Note that the left-hand side of \eqref{inlemma: hinge solution first stage requirement} is the absolute value of the first stage blip function or conditional treatment effect defined in \eqref{def: treatment effect stage 1}. Thus \eqref{inlemma: hinge solution first stage requirement} can be interpreted as a lower bound condition, indicating the minimum strength required for the first stage conditional treatment effect. Further implications of \eqref{inlemma: hinge solution first stage requirement} will be discussed after introducing Theorem \ref{theorem: hinge}, which demonstrates the necessity of  \eqref{inlemma: hinge solution first stage requirement} for the uniqueness of $\tilde d_1(H_1)$.
\begin{theorem}[$\tilde d_1$ and $\tilde d_2$ for hinge loss]
 \label{theorem: hinge}
  Suppose $\psi(x,y)=\min(x,y,1)$. Further, suppose Assumptions I-IV hold and  $Y_1$ and $Y_2$ are bounded below by some positive constant.

 \begin{itemize}
     \item \textbf{First stage:} If \eqref{inlemma: hinge solution first stage requirement} holds for some $h_1\in\H_1$, then $\tilde d_1(h_1)=d_1^*(h_1)$.  If \eqref{inlemma: hinge solution first stage requirement} does not hold, then $\tilde d_1(H_1)=\{1,-1\}$. 
     \item\textbf{Second stage:}
 If $h_2\equiv (h_1,a_1,y_1,o_2)\in\H_2$ is such that $a_1$ and $h_1$ satisfy $a_1=\tilde d_1(h_1)$, then $\tilde d_2(h_2)=d_2^*(h_2)$. For all other $h_2$,  $\tilde d_2(h_2)=\{-1,1\}$.
 \end{itemize}
\end{theorem}

 
  Theorem \ref{theorem: hinge} is proved in Supplement \ref{sup: hinge}. Its proof is based on straightforward algebra and elementary convex analysis results. The first observation from Theorem \ref{theorem: hinge} is that 
  the condition for $d_2^*(h_2)=\tilde d_2(h_2)$ is actually not restrictive. If the first stage treatment allocation follows  $\tilde d_1$, then $A_1=\tilde d_1(H_1)$, and hence  $\tilde d_2(H_2)$ matches with $d^*_2(H_2)$.  However,  the first stage appears to be more challenging for the hinge loss because when \eqref{inlemma: hinge solution first stage requirement} fails to hold, this loss is unable to discriminate between the two treatment strategies in the first stage. If $d_1^*(H_1)$ is unique, then $d_1^*$ and $\tilde d_1$ will disagree in such situations. 
As a trivial example of such a scenario, consider the illustration in \eqref{def: easy example}.  \label{page: Theorem 2 patho}
    In this case,  the absolute value of the first stage conditional treatment effect is five  but the threshold in the right hand side of \eqref{inlemma: hinge solution first stage requirement} is eight for all $H_1$. Thus   $d_1^*(H_1)$ is unique, and it is always $-1$ but \eqref{inlemma: hinge solution first stage requirement} does not hold in this example, thereby confirming Fisher inconsistency.   We provide more examples of the failure of  \eqref{inlemma: hinge solution first stage requirement} in Supplement \ref{sec: strange conditions}.  Given that \eqref{inlemma: hinge solution first stage requirement} represents a minimal strength condition for the first-stage conditional treatment effect, the above discussion indicates that the hinge loss requires a sufficiently strong first-stage conditional treatment effect to accurately identify the first-stage optimal treatment. 
\label{page: hinge loss statiic} Similar to many other concave losses, the univariate version of hinge loss is Fisher consistent for the single-stage problem \citep[][see also Figure \ref{Figure: 1d functions}]{zhao2012}. However, Fisher inconsistency of Hinge loss has been observed in  some classification problems involving more than two classes (see \cite{liu2007} for a detailed account). Hinge loss is also not Fisher consistent for maximum score estimation problem in linear binary response model  \citep{feng2022nonregular}. In our case, the inconsistency stems from the first-stage treatment assignment, which aligns with the previous examples of concave losses in Section \ref{sec: failure of smooth concave}. This happens because the final-stage (in our case the second-stage) treatment assignment in DTR resembles a single-stage weighted classification problem, where concave surrogates work. The inherent difficulty of DTR  manifests in the treatment assignments of the early stages. This is unsurprising because the early-stage treatment assignments need to take into account the  potential outcomes of all future stages.

Some additional remarks are pertinent concerning the location transformation employed to ensure the positivity of outcomes because the location transformation makes it more challenging to satisfy \eqref{inlemma: hinge solution first stage requirement}. To see this, consider a hypothetical situation where \eqref{inlemma: hinge solution first stage requirement} holds at $H_1=h_1$ for some data distribution. If we perform a location shift by transforming $Y_1$ to $Y_1+C$ and $Y_2$ to $Y_2+C$, the left-hand side of \eqref{inlemma: hinge solution first stage requirement} increases by $C$, while the right-hand side grows by $3C$. Therefore, if $C$ is large enough,  \eqref{inlemma: hinge solution first stage requirement}  will 
 no longer hold for the location-transformed data.  Given the positivity of $Y_1$ and $Y_2$ does not ensure Fisher consistency anyway,  one may question the form of $\tilde{d}$  when $Y_1$ and $Y_2$ are allowed to take non-positive values. We delve deeper into this topic  in Supplement \ref{sup: hinge}.


\begin{remark}
 \label{remark: SOWL}
  As previously mentioned, the SOWL method proposed by \cite{zhao2015} is based on  the bivariate hinge loss described in Theorem~\ref{theorem: hinge}. In their paper, it was claimed that the hinge loss always leads to $\tilde{d}=d^*$. Our analysis demonstrates that the agreement between $\tilde{d}$ and $d^*$ relies on the fulfilment of \eqref{inlemma: hinge solution first stage requirement} when $d^*_1(H_1)$ is unique. There exist distributions where \eqref{inlemma: hinge solution first stage requirement} holds for all $h_1\in\H_1$, resulting in  $\tilde{d}=d^*$, while other distributions violate \eqref{inlemma: hinge solution first stage requirement} for some $h_1\in\H_1$. From a high level, this condition requires the first stage conditional treatment effect to be larger than some threshold.  For specific examples and further elaboration, refer to Supplement \ref{sec: strange conditions}.
\end{remark}

}
 \subsection{Construction of Fisher consistent surrogates}
 \label{sec: fisher consistent surrogate}
 In this section, we construct  Fisher consistent loss functions for two stage DTR classification.
 Noting the connection between binary classification and DTR classification, we consider bivariate loss functions of form $\psi(x,y)=\phi_1(x)\phi_2(y)$ where $\phi_1$ and $\phi_2$ themselves are univariate loss functions. The most intuitive choice of $\phi_i$'s would be the Fisher consistent losses for  one stage DTR. However,  $\phi_i(x)=-\exp(-x)$ is Fisher consistent in one stage \citep{bartlett2006, chen2017} although the product $\phi_1(x)\phi_2(y)$ is inconsistent for the two stage setting (see Section~\ref{sec: concave}). The above indicates that $\phi_i$'s Fisher consistency is insufficient for $\psi$ to mimic the bivariate zero-one loss function effectively.
 
   In fact, our calculations hint that
 $\phi_2$ needs to share  a particular property of the zero-one loss function, that is for some constant $C > 0$, 
 \begin{align}
     \label{cond: sigmoid type}
    \sup_{x\in\RR}\slb \eta\phi_2(x)+(1-\eta)\phi_2(-x)\srb=C\max(\eta,1-\eta) .
 \end{align}
  The above property is satisfied by the sigmoid function, which is non-concave, and Fisher consistent for binary classification \citep{bartlett2006}.  
  Interestingly, \eqref{cond: sigmoid type} alone does not guarantee the fisher consistency of $\psi=\phi_1\phi_2$. \textcolor{black}{For instance, the  loss $\phi(x)=\min(x+1,1)$ satisfies \eqref{cond: sigmoid type} with $C=2$ \citep[cf.][]{bartlett2006} but $\psi(x,y)=\phi(x)\phi(y)$ is not Fisher consistent for DTR when the number of stages is more than two.  } Therefore \eqref{cond: sigmoid type} is not a sufficient for Fisher consistency.
  Now we introduce a sufficient condition for Fisher consistency.  
  \begin{condition}\label{cond: sigmoid  condition}
$\phi$ is a strictly increasing function such that
\begin{enumerate}
    \item $\phi(x)>0$ for all $x\in\RR$.
    \item  For all $x\in\RR$, $\phi(x)$ satisfies $\phi(x)+\phi(-x)=C_\phi$ where $C_\phi>0$ is a constant.
    \item $\lim_{x\to\infty}\phi(x)=C_\phi$ and  $\lim_{x\to-\infty}\phi(x)=0$.
\end{enumerate}
 \end{condition}
 We will show in the upcoming Theorem \ref{lemma: approximation error} that Condition \ref{cond: sigmoid  condition} is sufficient for Fisher consistency in the sense that if $\phi$ satisfies Condition \ref{cond: sigmoid  condition}, then $\psi(x,y)=\phi(x)\phi(y)$ is Fisher consistent. 
{\color{black} A\label{page: distributional presentration} $\phi$ satisfying  Condition~\ref{cond: sigmoid  condition} is Fisher consistent for binary classification, and it also satisfies  \eqref{cond: sigmoid type} (see  Lemma~\ref{lemma: sigmoid lemma} in Supplement~\ref{secpf: calibration lemma: auxilliary}). 
Notably, this $\phi$ possesses another important property.  When $C_\phi=1$ and $\phi$ is continuous, $\phi$ becomes the distribution function of an unbounded symmetric random variable. In contrast, the previously mentioned univariate hinge loss $\phi(x)=\min(x+1,1)$ lacks this property. Specifically, when smooth, $\phi$ can be perceived as a smooth version of the 0-1 loss, smoothed via a symmetric distributional kernel. Consequently, it can be inferred that surrogates satisfying Condition~\ref{cond: sigmoid condition} closely approximate the 0-1 loss.   That being said, we do not yet know if Condition~\ref{cond: sigmoid condition} is necessary for Fisher consistency in the DTR problem. }


We provide some examples of functions satisfying Condition~\ref{cond: sigmoid  condition}  below.  
  \begin{example}\label{Example: sigmoid}
The following odd functions are non-decreasing with range $[-1,1]$: 
\begin{enumerate}
    \item $f_a(x)=\frac{x}{1+|x|}$.
    \item $f_b(x)=\frac{2}{\pi}\arctan\lb\frac{\pi x}{2} \rb$.
    \item $f_c(x)=\frac{x}{\sqrt{1+x^2}}$.
    \item $f_d(x)=\tanh(x)$, where $\tanh(x)=\frac{e^x-e^{-x}}{e^x+e^{-x}}$.
\end{enumerate}
Then $\phi_\jmath(x)=1+f_\jmath(x)$ satisfies Condition~\ref{cond: sigmoid  condition} with $C_\phi=2$, for $\jmath = a, b, c, d$. See Figure~\ref{fig: sigmoid functions} for the pictorial representation of these functions.
\begin{figure}[h]
    \centering
     \begin{subfigure}[b]{\textwidth}
     \centering
    \includegraphics[width=\textwidth]{"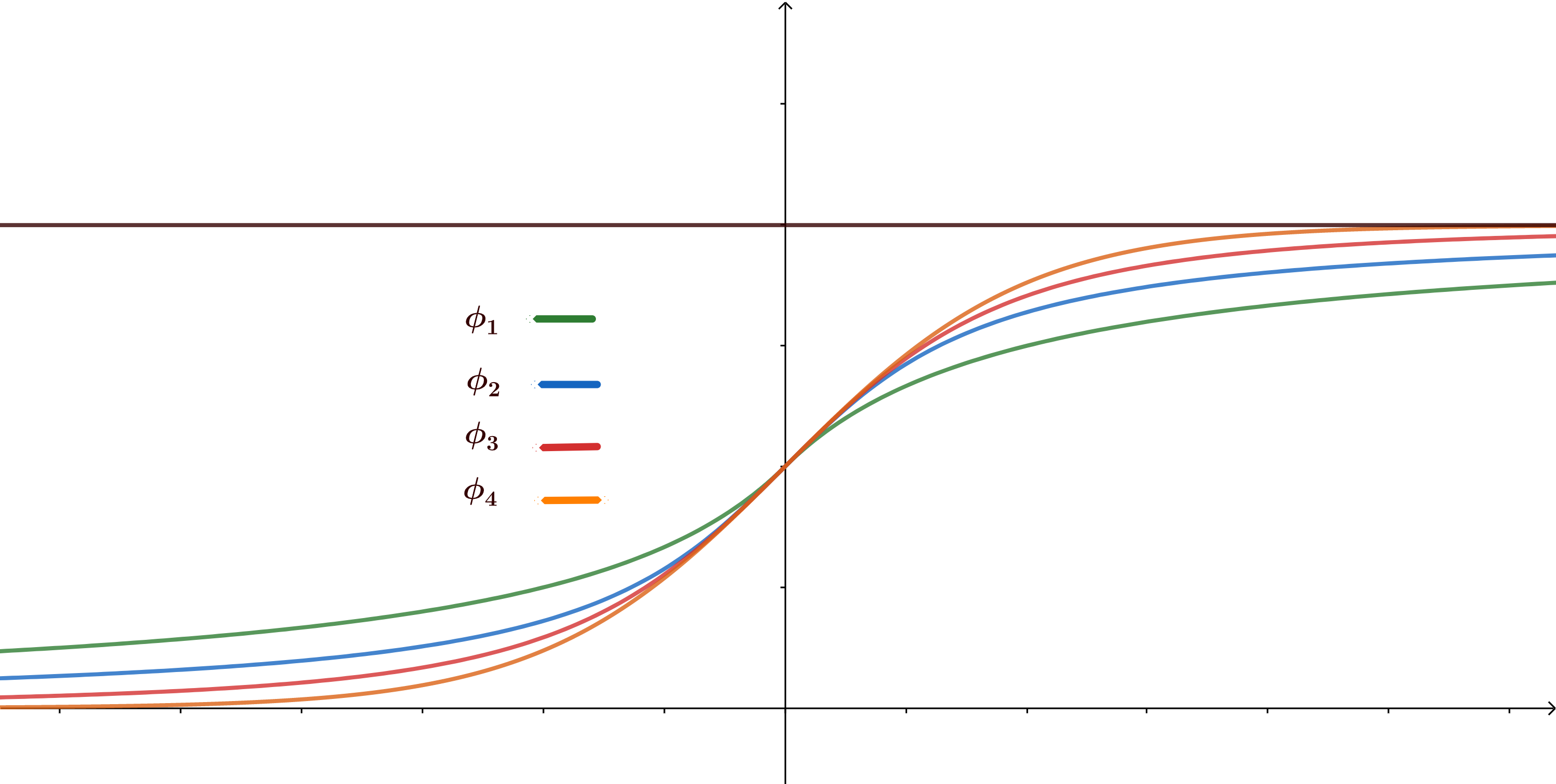"}
    \caption{Plot of  $\phi$'s.}\label{fig: sigmoid functions}
    \end{subfigure}\\
     \begin{subfigure}[b]{\textwidth}
     \centering
    \includegraphics[width=\textwidth]{"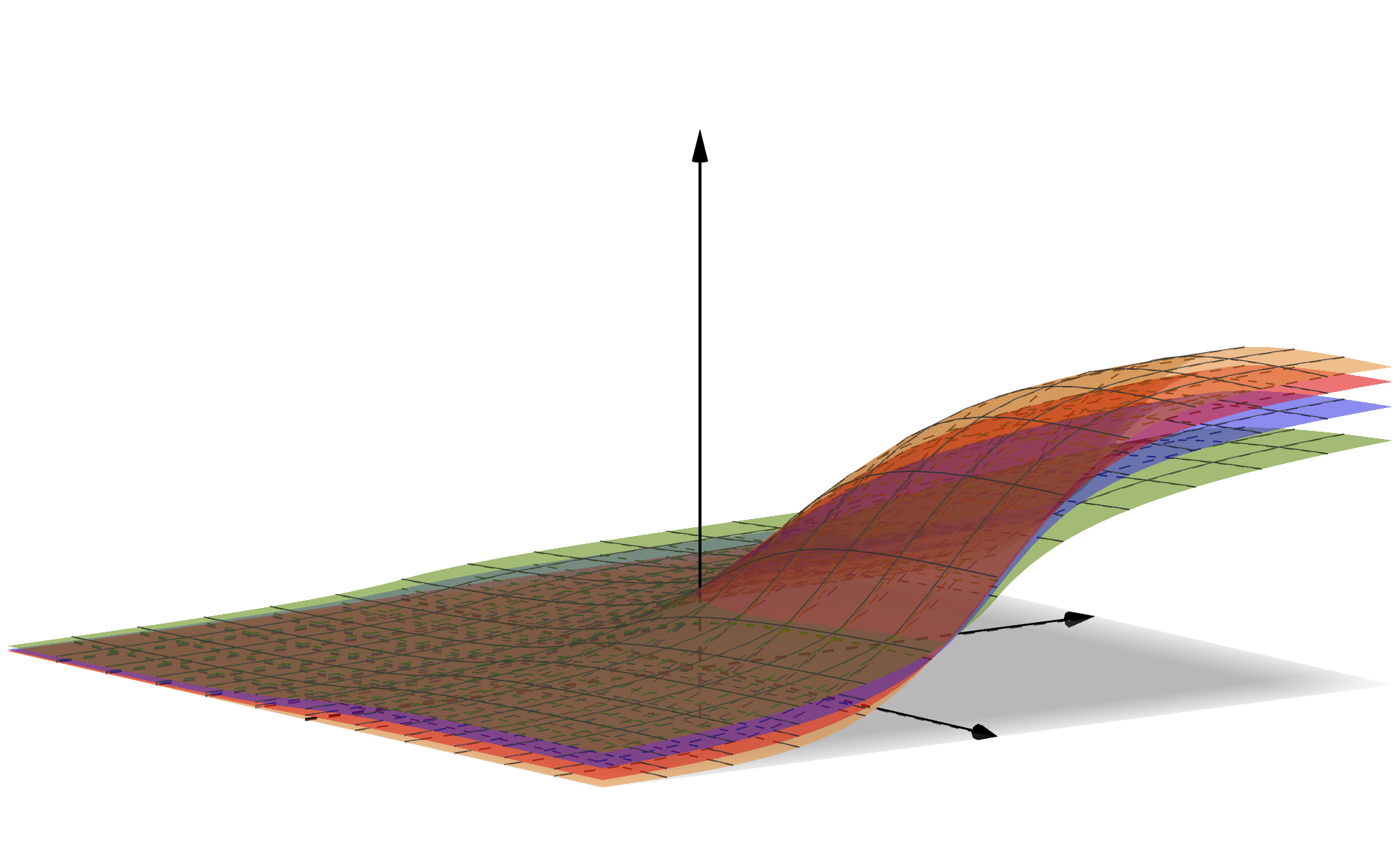"}
    \caption{Plot of the surrogate $\psi$'s obtained from the above $\phi$'s.}\label{fig: sigmoid functions biv}
    \end{subfigure}
    \caption{Plot of the $\phi$'s in Example \ref{Example: sigmoid} and the corresponding $\psi$'s. Here $\psi(x,y)=\phi(x)\phi(y)$.}
\end{figure}
\end{example}

{\color{black}
\label{page: diff from zo} Our approach involving non-concave surrogates leads to non-convex optimization problems, prompting the question of how it differs from directly optimizing the original value function. While both approaches lead to non-convex optimization, our method results in a smooth optimization problem. In contrast, direct maximization of the value function would lead to a discontinuous optimization problem with jump discontinuities. Moreover, the objective function resulting from the latter optimization problem is flat at the regions of continuity. 

Section \ref{sec: opt: summary} and Supplement \ref{sec:opt error} entails that the surface of our surrogate optimization problem exhibits favorable properties, and the optimization error with gradient descent-type algorithms may be small under certain  conditions. In contrast, gradient descent-type methods would likely fail for the discontinuous  problem resulting from direct value function maximization, and no known condition or method guarantees small optimization error for these methods in such problems \citep{xu2014model}.  This makes direct optimization of the value function considerably more challenging than our method. Objectives with 0-1 loss appear naturally in various machine-learning problems. As far as we know, In current statistical machine-learning literature, direct optimization of such objectives is avoided, and instead, the original 0-1 loss is replaced with a more well-behaved surrogate loss, whether convex or not, whenever such a surrogate is available  \citep{mukherjee2021optimal,xu2014model,horowitz1992smoothed,feng2022nonregular,pedregosa2017,gao2011,calauzenes2012}.

The class specified by\label{page: examples} Condition \ref{cond: sigmoid condition} has been mentioned in various machine-learning problems, often presented in forms appropriate for a minimization problem. In certain instances, it is referred to as the smoothed 0--1 loss. In some of these machine learning  problems, this class has been proposed in situations where convex surrogates have demonstrated inconsistency.   For example, \cite{gao2011} has shown that this class of surrogates is Fisher consistent for multilabel classification with ranking loss, where convex surrogates are inconsistent. \cite{feng2022nonregular} established Fisher-consistency-related guarantees for such surrogates  in a diverse range of problems including  Covariate-adjusted Youden index estimation, one-bit compress sensing, and maximum score estimation in  binary response model; see also \cite{xu2014model,mukherjee2021optimal}. Especially for maximum score estimation, \cite{feng2022nonregular}  showed that common convex surrogates such as exponential and hinge loss are inconsistent \citep{feng2022nonregular}.  Finally, the surrogate loss used for multivariate $\psi$-learning in the context of multicategory classification is a non-smooth member of our class \citep{liu2006multicategory}. The authors of that work claim that this non-concave surrogate outperforms SVM, which relies on hinge loss. }

\subsubsection{Fisher consistency of $\psi$}
Instead of directly proving Fisher consistency, we will  bound the true regret $V^*-V(f_1,f_2)$ in terms of  the $\psi$-regret $V_\psi^*-V_\psi(f_1,f_2)$. 
The benefit of such a bound is that the rate of convergence of the true regret will be readily given by that of the $\psi$-regret, which we actually minimize.

  As mentioned earlier, the true regret and the $\psi$-regret parallel the  excess risk $\mathcal R(f)-\R^*$ and the $\phi$-excess risk  $\mathcal R_{\phi}(f)-\R_\phi^*$ in binary classification. The relationship between the latter have been well-studied.  For  Fisher consistent $\phi$, \cite{bartlett2006} show that
 \[h_\phi\slb\mathcal R(f)-\R^*\srb\leq \mathcal R_{\phi}(f)-\R_\phi^*,\]
 where $h_\phi$ is a  convex function satisfying $h_\phi(0)=0$. In view of the  fact that the univariate sigmoid loss  leads to a linear $h_\phi$   \citep[][Example 4]{bartlett2006}, it is reasonable to expect that a similar inequality holds when
 $\psi(x,y)=\phi(x)\phi(y)$ with $\phi$  as in Condition~\ref{cond: sigmoid  condition}, as confirmed in following theorem.
 
 
 

 \begin{theorem}\label{lemma: approximation error}
 Suppose $Y_1,Y_2>0$ and Assumptions I-IV hold. Let $\psi(x,y)=\phi(x)\phi(y)$ with $\phi$ satisfying Condition \ref{cond: sigmoid  condition} with some $C_\phi>0$. Then 
 \begin{equation}\label{scaling: ub: regret}
     V^*-  V(f_1,f_2)\leq \frac{ \slb V_\psi^*-  V_\psi(f_1,f_2)\srb}{(C_\phi/2)^2}.
 \end{equation}
 \end{theorem}
 Theorem~\ref{lemma: approximation error} immediately implies Fisher consistency because if $V_{\psi}(f_{1n},f_{2n})$ converges to $ V_\psi^*$ for some $(f_{1n}$, $f_{2n})\in\F$, then $V(f_{1n},f_{2n})\to V^*$ as well. Theorem~\ref{lemma: approximation error} is proved in Supplement~\ref{secpf: fisher consistency lemma}. 
 
{\color{black} As\label{intuition}  mentioned earlier, a necessary requirement for Fisher consistency is an agreement between $\tilde d$ and $d^*$. 
Proving the latter is also a key step in the proof of Theorem \ref{lemma: approximation error}. Let us provide some intuition as to why the $\tilde d$ corresponding to our $\psi$ may agree with $d^*$.

We mentioned earlier that any $\phi$ satisfying Condition \ref{cond: sigmoid  condition} is Fisher consistent for binary classification. It can be  shown that Fisher consistency for binary classification translates to Fisher consistency for the single-stage case under Assumptions I-IV \citep{chen2017}. Using this insight, we can show that the second stage treatment allocation $\tilde d_2(H_2)$ matches with $d_2^*(H_2)$ for our $\psi$. 
Regarding the first stage, after some algebraic manipulation, we can demonstrate that $\tilde d_1(H_1)$ takes the form in \eqref{eq: prelude to surrogate} analogous to the exponential loss, but with $h(x,y) = \max(x, y)$.  This particular form of $h$ is primarily driven by \eqref{cond: sigmoid type} and the positivity of $\phi$. Since $d_1^*$ satisfies \eqref{eq: prelude to surrogate} with $h(x,y) = \max(x, y)$, the above leads to  $\tilde d_1=d_1^*$.} 

 {\color{black} \label{page: binary}We want to remind the readers that the assumption $Y_1, Y_2 > 0$ is not restrictive. As mentioned earlier, in cases where the observed outcomes are not positive, a location transformation can be applied to ensure positivity without altering the optimal treatment policy $d^*$ and, consequently, $\tilde{d}$.  We also want to emphasize that  Theorem \ref{lemma: approximation error}, as well as  all our upcoming theorems, do not distinguish between continuous and discrete outcomes. Therefore, our method and theory apply to discrete and binary outcomes, which are of interest in many applications.}

 \begin{remark}[Scaling of the $\phi$'s]
 The scaling factor $C_\phi/2$ appears in the regret of (\ref{scaling: ub: regret}) because $\phi$ differ from the zero-one function in scale by a factor of $C_\phi/2$. 
To understand the impact of the scaling factor in the regret bound, suppose  $\phi_2=a\phi_1$ for some $a>0$, and $\psi_t(x,y)=\phi_t(x)\phi_t(y)$ for $t=1,2$. Then
\[ \frac{ \slb V_{\psi_1}^*-  V_{\psi_1}(f_1,f_2)\srb}{C_{\phi_1}^2/4}= \frac{ \slb V_{\psi_2}^*-  V_{\psi_2}(f_1,f_2)\srb}{C_{\phi_2}^2/4}.\]
 Thus,  the regret bound in \eqref{scaling: ub: regret} does not depend on the scale of $\phi$. Nevertheless, during our implementation, we take the scaling factor $C_\phi/2$ to be one so that the surrogate loss is at the same scale as the original zero-one loss.
 \end{remark}

{\color{black}  We\label{non-concave 2} would like to emphasize a crucial point. While our algorithm is capable of handling large sample sizes, it is important to note, as we will discuss in Supplement \ref{sec:opt error}, that guarantees regarding its convergence to the global maximum are scarce. This limitation is a common challenge encountered in non-concave optimization problems. However, it is important to recognize that we employ non-concave losses due to the apparent absence of Fisher-consistent concave losses. In other words, non-concave optimisation may be the only viable choice if one aims to solve the DTR problem through simultaneous optimization. This highlights the inherent difficulty of the DTR problem.

To circumvent non-concave optimization while retaining theoretical guarantees, one has two options: employing a stage-wise Fisher-consistent optimization method like BOWL or opting for a regression-based approach such as Q-learning. However, it is important to note that BOWL achieves Fisher consistency at the expense of reduced sample size in the first stage \cite{zhao2015}. Our simulations in Section \ref{sec: simulation} indicate that BOWL does not outperform our proposed method. Additionally, our simulations demonstrate that BOWL exhibits significantly longer runtimes compared to our method for large sample sizes.}

\section{Main methodology}
\label{sec: main methodology}
In this section, we describe how we use the Fisher-consistent surrogate derived in Section~\ref{sec: fisher consistent surrogate} to estimate the optimal treatment regimes. For the remainder of this paper except Supplement \ref{sec: add hinge loss}, unless otherwise mentioned, $\phi$ will denote a univariate surrogate satisfying Condition~\ref{cond: sigmoid condition}, and $\psi$ will denote the bivariate surrogate $\psi(x,y)=\phi(x)\phi(y)$ where $\phi$ satisfies Condition~\ref{cond: sigmoid condition}.
Define the empirical    $\psi$-value function
\begin{equation}\label{def: value function sample estimate}
\widehat V_{\psi}(f_1,f_2)=\mathbb P_n \lbt\dfrac{(Y_1+Y_2)\psi\slb A_1 f_1(H_1), A_2f_2(H_2)\srb}{\pi_1(A_1\mid H_1)\pi_2(A_2\mid H_2)} \rbt
\end{equation}
Because $\PP$ is unknown, we maximize $\widehat V_{\psi}(f_1,f_2)$
instead of $V_\psi(f_1,f_2)$. Ideally, one should maximize $\widehat V_{\psi}(f_1,f_2)$ over $\F$ but  brute force search over $\F$ is impossible unless $\H_1$ and $\H_2$ are discrete spaces with finite cardinality. Therefore, in practice, one may 
 optimize $\widehat V_\psi(f_1,f_2)$ over a nested  class 
 \[\mathcal U_1\subset\ldots\subset \mathcal \U_n\subset \mathcal F,\]
where $\U_n$ is some rich class of classifiers, preferably a universal class \citep[see][]{zhang2018}. We will discuss  them in more detail later in Section~\ref{sec: Approximation error}. Whatever is the choice of $\U_n$,  maximization of $\widehat V_\psi(f_1,f_2)$ over $(f_1,f_2)\in\U_n$ generally leads to a non-convex optimization problem. 


The surrogate loss based DTR optimization allows flexibility in the choice of $\U_n$ and the modification of the empirical loss $\widehat V_\psi(f_1,f_2)$ to accommodate high dimensional covariates, non-linear effects, and variable selection. 
 One can maximize $\widehat V_\psi(f_1,f_2)+\mathcal P(f_1,f_2)$ instead of $\widehat V_\psi(f_1,f_2)$ to enable variable selection and attain stable estimation, where $\mathcal P(f_1,f_2)$ is a penalty term. One can include complex basis functions in $\U_n$ to incorporate non-linear effects. 
 %
 For example, tree and list based methods \citep{zhang2018interpretable, sun2021stochastic,laber2015tree} {\color{black} as well\label{page: missing REF} as neural networks (see Section \ref{ex; neural network} for details) can be potentially adapted to construct $\U_n$.} 
 Moreover, our method can be extended to $K$ stages by taking $\psi(x_1,\ldots,x_k)=\prod_{i=1}^k \phi(x_i)$.

 \subsection{Decomposition of errors}
\label{sec: error decomposition}
In this section, we will discuss the decomposition of $\psi$-regret of DTRESLO into three sources of errors. 
To that end, let us denote our classifiers by $(\hfa, \hfb)$. 
We will provide upper bounds for the $\psi$-regret $V_\psi(\hfa,\hfb)$, which readily produces an upper bound for the true regret $V(\hfa,\hfb)$ by Theorem~\ref{lemma: approximation error}. 
Before going into further detail, we point out that
\[V(\hfa,\hfb)=\PP\lbt (Y_1+Y_2)\frac{1[A_1\phi(\hfa(H_1))>0]1[A_2\phi(\hfb(H_2))>0]}{\pi_1(A_1\mid H_1)\pi_2(A_2\mid H_2)} \rbt\]
is a random quantity because here we assume that $H_t,A_t,Y_t$ $(t=1,2)$ are  drawn from $\PP$  independent of $(\hfa,\hfb)$. The same holds regarding the $\psi$-regret $V_\psi(\hfa,\hfb)$.
We decompose the $\psi$-regret according to three sources of errors: (i) approximation error due to the approximation of $\F$ by $\U_n$; (ii) estimation error due to the use of finite sample; and (iii) optimization error due to the possibility of not achieving global maximization for $\widehat V_\psi(f_1,f_2)$  since $\psi$ is non-concave.   We define the optimization error  as
\[\Opn=\sup_{(f_1,f_2)\in\U_n}\widehat V_\psi(f_1,f_2)-\widehat V_\psi(\hfa,\hfb).\]

We first provide some heuristics and intuitions for the error decomposition. 
For the time being, let us assume that $\argmax_{(f_1,f_2)\in\U_n}V_\psi^*(f_1,f_2)$ is attained at some $(\tha,\thb)\in\U_n$. The existence of $(\tha,\thb)$ is not guaranteed in general, and even if they exist, $(\tha,\thb)$ will be hard to characterize for an arbitrary $\U_n$. We thus do not assume the existence of $(\tha,\thb)$ in our proof. We define the map
 $\xi_{f_1,f_2,g_1,g_2}:\H_1\times\O_2\times \RR^2\times\{\pm 1\}^2\mapsto\RR$ by
\begin{align}\label{def: GE: xi function}
  \MoveEqLeft  \xi_{f_1,f_2,g_1,g_2}(\mathcal D)\nn\\
    :=&\ \frac{(Y_1+Y_2)\left\{ \psi( A_1g_1(H_1),A_2 g_2(H_2))-\psi( A_1 f_1(H_1),A_2 f_2(H_2)) \right\}}{\pi_1(A_1\mid H_1)\pi_2(A_2\mid H_2)}.
\end{align}
Elementary algebra shows that the $\psi$ regret can be decomposed as follows:
 \begin{align}\label{regret decomposition}
   \MoveEqLeft  V_\psi^*-V_\psi(\hfa,\hfb)\nn\\
     =&\ \underbrace{V_\psi^*-V_\psi(\tha,\thb)}_{\text{Approximation error }}+ (V_\psi-\widehat V_\psi)(\tha,\thb)-(V_\psi-\widehat V_\psi)(\hfa,\hfb)\nn\\
     &\ + \widehat V_\psi(\tha,\thb)- \sup_{(f_1,f_2)\in\U_n}\widehat V_\psi(f_1,f_2)+\underbrace{\sup_{(f_1,f_2)\in\U_n}\widehat V_\psi(f_1,f_2)-\widehat V_\psi(\hfa,\hfb)}_{\text{Optimization error: }\Opn}\nn\\
      \leq &\ \text{Approximation error }+\underbrace{|(\P_n-\P)[\xi_{f_1,f_2,\tha,\thb}]|}_{\text{Estimation error }}+\Opn
 \end{align}
Clearly, $\Opn$ depends on the optimization method used to maximize $\widehat{V}_\psi(f_1,f_2)$ over $\U_n$. {\color{black} We\label{page: error decomposition}  study the optimization error in Supplement \ref{sec:opt error} for specific scenarios. The primary emphasis of this paper revolves around the estimation error and the approximation error. }
In our sharp analysis of the $\psi$-regret, the estimation error bound depends on the approximation error in an intricate manner; see Supplement \ref{sec:subtlety_approx_est_error} for more details.  
To keep our presentation short and focused, therefore, we discuss the approximation error in  Section~\ref{sec: Approximation error}, and  present the final regret bound  in  Section~\ref{sec: generalization}. Details on the explicit analysis of the estimation error can be found in Supplement~\ref{sec: proof of gen error}. Finally, owing to the potential non-concavity, the sharp analysis of the $\psi$-regret is significantly more subtle than existing results and approaches in the literature. 
We elaborate on this more in a detailed discussion presented in Supplement \ref{sec:subtlety_approx_est_error}.


 
 In what follows, similar to \cite{zhao2015}, we assume that the propensity scores, i.e. $\pi_1$ and $\pi_2$ are known. This will hold in particular under  a clinical trial like SMART \citep{laber2019}, but not for observational data.
When $\pi_1$ and $\pi_2$ are unknown, they can be estimated using a logistic regression model. This additional estimation step will not change the approximation error but the estimation error will likely change. 

 \section{Approximation error}
 \label{sec: Approximation error}

\subsection{Assumptions}
\label{sec: small noise}
 
To establish the  convergence rate   of the approximation error, we require two assumptions.
First, we require the standard assumption that the outcomes are bounded.
\begin{assumption}\label{assump: bound on Y}
Outcomes $Y_1,Y_2$ satisfy $\max (Y_1,Y_2)\leq C_y$.
\end{assumption}
The second assumption is the DTR version of Tsybakov's small noise assumption \citep{tsybakov2004, audibert2007}. Recall the blip functions/ conditional treatment effects $\T_1$ and $\T_2$ defined in \eqref{def: treatment effect stage 1} and \eqref{def: treatment effect stage 2}, respectively. 
Because we assume $Y_1$ and $Y_2$ are bounded away from zero, $\eta_t(H_t)=1/2$ if and only if $\T_t(H_t)=0$ for $t=1,2$. Therefore, the treatment boundary $\{h_t:\eta_t(h_t)=1/2\}$ can also be formulated as  $\{h_t:\T_t(h_t)=0\}$. 

In  classification literature, it is well-noted that obtaining a fast rate of convergence (faster than $n^{-1/2}$) requires control on the distribution of the random variable $\eta(X)-1/2$ near the decision boundary $\{x:\eta(x)=1/2\}$ to some degree, which gives rise to the so-called  margin conditions \citep{tsybakov2004, audibert2007}. Similarly, in the DTR context, even with regression-based methods, regulation near the conditional treatment effect boundary $\{H_t:\T_t(H_t)=0\}$
are generally required; see Appendix~\ref{sec: discussion: small noise} for a detailed discussion.
Thus it is expected that we too would require  control on the rate of decay of $\G-1/2$ and $\Gt-1/2$ near the treatment boundary $\{h_1:\G(h_1)=1/2\}$ and $\{h_2:\Gt(h_2)=1/2\}$, respectively, to obtain sharp bound on the $\psi$-regret.
Among  many variants of margin condition,  we consider the Tsybakov small noise condition (Assumption MA of \cite{audibert2007}; see also Proposition 1 of \cite{tsybakov2004}), which has seen wide use in the literature \citep{audibert2007, steinwart2007, blanchard2008}. The DTR formulation of Tsybakov  small noise condition takes the following form:
\begin{assumption}[Tsybakov small noise assumption]
\label{assumption: small noise}
There exist a constant $C>0$, a small number $t_0\in(0,1)$, and positive reals $\alpha_1,\alpha_2$ such that
\[P( 0<|\eta_1(H_1)-1/2|\leq t)\leq Ct^{\alpha_1},\quad  P( 0<|\eta_2(H_2)-1/2|\leq t)\leq C t^{\alpha_2}\]
for all $t<t_0$.
\end{assumption}
The parameters $\alpha_1$ and $\alpha_2$ are the \emph{Tsybakov noise exponents}. We already noted that the $Y_i$'s are bounded below.
 Since the outcomes are also bounded above by Assumption~\ref{assump: bound on Y}, 
 Assumption~\ref{assumption: small noise} is equivalent to saying
\begin{equation}\label{def: small noise: treatment effect}
    P(0<\mathcal T_1(H_1)<t)+P(0<\mathcal T_2(H_2)<t)\leq Ct^{\alpha},\quad\text{ for all }t\leq t_0.
\end{equation}
This alternative version is  more common in precision medicine literature \citep{qian2011, luedtke2016}. See Supplement \ref{sec: discussion: small noise} for more details on the small noise assumption or similar assumptions in precision medicine literature. 
Finally, observe that if a stage has Tsybakov noise exponent $\alpha$, then it also has noise exponent $\alpha'$ for all $\alpha'<\alpha$.  
Thus, to keep our calculations short, we assume that both stages have noise exponent $\alpha$ where $\alpha=\min(\alpha_1,\alpha_2)$. We postpone further discussion on Assumption \ref{assumption: small noise} till Supplement \ref{sec: discussion: small noise}.

Under Assumption~\ref{assumption: small noise}, it turns out that, our surrogates satisfying Condition~\ref{cond: sigmoid  condition} do not exhibit identical approximation error. The difference in the rate stems from the difference in their respective derivatives. Thus, it will be convenient to split the above-mentioned surrogates into two types.
\begin{definition}
\label{def: type of phi}
We say a surrogate $\phi$ satisfying Condition~\ref{cond: sigmoid  condition} is of type A if there exists a constant $B_\phi>0$ and $\kappa\geq 2$ such that $|\phi'(x)|<B_\phi(1+|x|)^{-\kappa}$ for all $x\neq 0$.
We say a surrogate $\phi$ satisfying Condition~\ref{cond: sigmoid  condition} is of type B if there exists a constant $B_\phi>0$ and $\kappa>0$ such that $|\phi'(x)|<B_\phi \exp(-\kappa|x|)$ for all $x\neq 0$.
\end{definition}

\begin{table}[]
    \centering
    \begin{tabular}{lccc}
      $\phi(x)$  &   Type & $B_\phi$ & $\kappa$\\
       \hline\hline
     (a)  $x/(1+|x|)+1 $ & A & 1 & 2\\
      (b) $\frac{2}{\pi}\arctan(\pi x/2)+1$ & A & 2 & 2\\
     (c)  $x/\sqrt{1+x^2}+1$ & A & $2^{3/2}$ & $3$\\
      (d) $1+\tanh(x)$ & B & 4 & $2$\\
       \hline
    \end{tabular}
    \caption{Caption}
    \label{tab: phi table}
\end{table}

First, Definition~\ref{def: type of phi} assumes $\phi$ to be smooth everywhere except perhaps at the origin. This restriction rules out non-smooth $\phi$'s, but they are uninteresting from our implementation perspective anyways. All the $\phi$'s we have considered in Example~\ref{Example: sigmoid} are differentiable at $\RR/\{0\}$ (see  Table~\ref{tab: phi table}; more details can be found in  Supplement \ref{sec: Verify table phi table}). Second, type A merely means  $\phi'$ decays polynomially in $|x|^{-\kappa}$, where type B $\phi$'s enjoy exponential decay of the derivative. 

\subsection{Approximation Error Rate}
\label{sec: approx-rate}

Theorem \ref{thm: approximation error} summarizes the approximation error rate. 
\begin{theorem}
\label{thm: approximation error}
Suppose $\mathbb P$ satisfies Assumptions I-IV, Assumption~\ref{assump: bound on Y} and Assumption \ref{assumption: small noise} with small noise coefficient $\alpha$. Let $0< a_n\to\infty$ be any sequence of positive reals. 
Further suppose there exist a small number $\delta_n\in(0,1)$ and maps $\tfa:\H_1\mapsto\RR$ and $\tfb:\H_2\times\{0,1\}\mapsto\RR$ so that
\[
    \|\tfa-(\G-1/2)\|_{\infty}+\|\tfb-(\Gt-1/2)\|_{\infty}\leq \delta_n
\]
where $\G$ and $\Gt$ are defined in \eqref{def: G} and \eqref{def: Gt}.
 Then for any $\phi$ of type A, the following holds for any $\alpha'\in(0,\alpha)$  satisfying $\alpha-\alpha'<1$:
\begin{align*}
\MoveEqLeft V_\psi^*-V_\psi(a_n\tfa,a_n\tfb)\\
  \lesssim &\begin{cases}
   a_n^{1-\kappa}+\min(\delta_n^{2+\alpha}a_n,\delta_n^{1+\alpha}) & \text{if }\kappa<2+\alpha\\
    \frac{a_n^{-\frac{1+\alpha}{1+(\alpha-\alpha')/(\kappa-1)}}}{\alpha-\alpha'}+\min(\delta_n^{2+\alpha}a_n,\delta_n^{1+\alpha})+\frac{\delta_n^{\alpha'+2-\kappa}}{(\alpha-\alpha')a_n^{\kappa-1}} & \text{if }\kappa\geq 2+\alpha.
   \end{cases}
\end{align*}
Suppose $\phi$ is of type B. Then
 \[V_\psi^*-V_\psi(a_n\tfa,a_n\tfb)\lesssim \frac{(\log a_n)^{1+\alpha}}{a_n^{1+\alpha}}+\min(a_n\delta_n^{2+\alpha},\delta_n^{1+\alpha})+a_n\delta_n\exp(-\kappa a_n \delta_n/2).\]
\end{theorem}

Theorem~\ref{thm: approximation error} entails that if $\{\tilde h_{n,1}, \tilde h_{n,2}\}$ approximates $\{\eta_1-1/2,\eta_2-1/2\}$ well in the sup-norm, then their scaled versions $\tilde f_{n,1}=a_n\tfa$ and $\tilde f_{n,2}=a_n\tfb$ incur small regret. 
 It may appear a bit unusual in that we require $\tilde f_{n,i}$'s to be close to  the functions $\eta_i-1/2$'s, where $V_\psi$ is actually maximized at $(\tilde f_1,\tilde f_2)$ (see Lemma~\ref{lemma: calibration lemma}). To that end, note that  the extended real valued functions $\tilde f_i$'s can not be approximated by any real valued $f_i$'s  because  $\|f_i-\tilde f_i\|_{\infty}$ is infinity for all such  $f_i$'s. However, the proof of Theorem~\ref{thm: approximation error} ensures that $a_n(\eta_i-1/2)$'s are  good proxy for
 the $\tilde f_i$'s because 
 \[V_\psi(\tilde f_1,\tilde f_2)-V_\psi\slb a_n(\G-1/2),a_n(\Gt-1/2)\srb\]
 is  small. The bounds in Theorem~\ref{thm: approximation error} holds for any small $\delta_n$, whose optimal rate will be found during from  the estimation error calculation.
 
 \subsection{Special case: strong separation with $\alpha = \infty$}
We next describe the convergence rates under a special case of  $\alpha=\infty$, which will be referred as the \emph{strong separation}:
\begin{assumption}[Strong separation]
\label{assump: boundedness}
$\G$ and $\Gt$ are bounded away from zero on their respective domains.
\end{assumption}
Under this setting, the conditional treatment effects are uniformly bounded away from zero. See Section~\ref{sec: discussion: small noise} for related discussion. This special case is of great interest in the current literature, cf.  \cite{zhao2012, zhao2015, qian2011}.
When the strong separation holds, the approximation error can be made much smaller than that of  Theorem~\ref{thm: approximation error} as detailed in
Proposition \ref{prop: approximation error: strong separation}. 
\begin{proposition}
\label{prop: approximation error: strong separation}
 Suppose the conditions of Theorem~\ref{thm: approximation error} hold except  $\P$ satisfies   Assumption~\ref{assump: boundedness} instead of Assumption~\ref{assumption: small noise}. Suppose $\tfa:\H_1\mapsto\RR$ and $\tfb:\H_2\times\{0,1\}\mapsto\RR$ satisfy
\[
    \|\tfa-(\G-1/2)\|_{\infty}+\|\tfb-(\Gt-1/2)\|_{\infty}\leq c
\]
for some $c>0$.
 Then there exists $C>0$ so that the following assertion holds for any large positive number  $a_n$:
\begin{align*}
 V_\psi^*-V_\psi(a_n\tfa,a_n\tfb)
  \lesssim\begin{cases}
   a_n^{-(\kappa-1)} & \text{if }\phi\text{ is of type A,}\\
 \exp(-\kappa C a_n) & \text{if }\phi\text{ is of type B}.
   \end{cases}
\end{align*}
\end{proposition}

\begin{remark}
Proposition~\ref{prop: approximation error: strong separation} implies that under the strong separation assumption,  $\tfa$ and $\tfb$ do not even need to approximate $\eta_1-1/2$ and $\eta_2-1/2$  very precisely and the approximation error decays substantially faster. To get a sense of how fast the regret diminishes, we consider the case when $a_n=n$. In this case, the regret  is $O(1/n)$ for a $\phi$ of type A because  $\kappa\geq 2$ for a type A $\phi$. The regret decays exponentially fast for a type B $\phi$, which can be attributed to the exponential decay of its  derivatives. 
 Recall that the derivative of the zero-one loss function  is exactly zero at any  $x\neq 0$. Thus  the type B surrogates bear closer resemblance to the original zero-one loss function. 
\end{remark}

Theorem~\ref{thm: approximation error} or Proposition \ref{prop: approximation error: strong separation} bound the approximation error because if $\U_n=\U_{1n}\times\U_{2n}$ is  such that 
\begin{align}
\label{intext: approx}
  \inf_{f_t\in\U_{tn}}\|f_t-(\eta_t-1/2)\|_\infty<\delta_n, \quad t=1,2,
\end{align}
then Theorem~\ref{thm: approximation error} or Proposition \ref{prop: approximation error: strong separation}  upper bound 
 $V_\psi^*-\sup_{(f_1,f_2)\in\U_n}V_\psi(f_1,f_2)$.

 \section{Estimation error}
\label{sec: generalization}

In this section, we focus on  the estimation error in \eqref{regret decomposition}, and provide sharp regret-bound for a selected set of classifiers by combining all sources of error. We assume that $(\hfa,\hfb)\in \U_n=\Ha\times \Hb$, where $\Ha,\Hb$ are  classes of functions. 
Our analysis in this section is  fully nonparametric because our $\U_n$ is agnostic of the underlying data-generating mechanism.
We first present some theorems (Theorems \ref{thm: weak: general thm}, \ref{theorem: GE: strong separation} and  \ref{theorem: GE: strong GE: basis}) for general $\U_n$'s. 
Then we will move  to study the particular examples of neural networks and wavelets.  


\subsubsection{Estimation error  when $\U_n$ is a general function-class}

For general function-classes, we  need some assumptions to control the complexity of $\U_n$. Such assumptions are widely used for bounding the expectation of the estimation error \citep{koltchinskii2009, bartlett2006, bartlett2002}.
To define complexity in the context of function-classes, we need to introduce the concept of the bracketing entropy. Given two functions $f_l$ and $f_u$, the bracket $[f_l,f_u]$ is the set of all function $f$ satisfying $f_l\leq f\leq f_u$. Suppose  $\|\cdot\|$ is a norm on the function-space  and $\e>0$. Then $[f_l,f_u]$ is called  an $\e$-bracket  if $\|f_u-f_l\|<\e$. For a function-class $\mathcal G$, we define the bracketing entropy $N_{[\ ]}(\e,\mathcal G,\|\cdot\|)$ to be the minimum number of $\e$-brackets needed to cover $\mathcal G$. This  is a measure of the complexity of $\mathcal G$.
We will see that the estimation error directly depends on the bracketing entropy  of $\U_n$. We will first consider the case when just the small noise assumption (Assumption \ref{assumption: small noise}) holds, and then we will move to the special case when the strong separation (Assumption~\ref{assump: boundedness}) holds. 

We derive the estimation error of DTRESLO under the small noise assumption  (Assumption \ref{assumption: small noise}) when
\begin{equation}\label{instatement: GE: weak GE}
    N_{[\ ]}(\e,\U_{tn}, \|\cdot\|_\infty)\lesssim \lb\frac{A_{n}}{\e}\rb^{\rho_n},\quad t=1,2
\end{equation}
where $A_n,\rho_n>0$.  This leads to the regret bound of Theorem \ref{thm: weak: general thm}  that depends on $A_n$ and $\rho_n$.
The $\U_n$'s that satisfy \eqref{instatement: GE: weak GE} are called VC-type classes \citep[p. 41][]{koltchinskii2009}. In all our examples, $\U_n$ will  satisfy \eqref{instatement: GE: weak GE} for appropriate $A_n$ and $\rho_n$.

 \begin{theorem}
\label{thm: weak: general thm}
Suppose $\U_n$  is such that  there exists $A_n>0$ and $\rho_n\in\RR$ so that  \eqref{instatement: GE: weak GE} holds with $\liminf_n\rho_n>0$, $\rho_n\log A_n=o(n)$, and  $\liminf_n\rho_n\log A_n>0$. 
Further suppose there exist $(\tha,\thb)\in\mathcal U_n$
so that
\begin{equation}\label{instatement: GE: weak: approx}
    \|\tha/a_n-(\G-1/2)\|_\infty+\|\thb/a_n-(\Gt-1/2)\|_\infty\leq \lb \frac{\rho_n\log A_n}{n}\rb^{1/(2+\alpha)}
\end{equation}
for some $a_n=n^a$ where $a>1$.
We also assume that $\mathbb P$ satisfies Assumptions I-IV, Assumption \ref{assump: bound on Y}, and Assumption \ref{assumption: small noise} with coefficient $\alpha>0$. Then there exist $C>0$  and $N_0\geq 1$ such that for all $n\geq N_0$ and all $x>0$,
\[V_\psi^*-V_\psi(\hfa,\hfb)\leq  C\max\lbs (1+x)^2 (\log n)^2 \lb \frac{\rho_n\log A_n}{n}\rb^{\frac{1+\alpha}{2+\alpha}},\Opn\rbs\]
with probability at least $1-\exp(-x)$.
\end{theorem}
Theorem \ref{thm: weak: general thm} is proved in Section~\ref{sec: proof of gen: weak sep}. 
In our examples with neural networks and wavelets, we will see that the regret bound in Theorem~\ref{thm: weak: general thm} leads to sharp rates provided $\G$ and $\Gt$ satisfy some smoothness conditions. 

Next we will derive  regret bounds under the special case of strong separation (Assumption~\ref{assump: boundedness}). We begin with a theorem that consider $\U_n$'s satisfying the following entropy bound:
\begin{align}\label{entropy bound: large class}
    \log N_{[\ ]}(\e,\U_{tn}, L_2(\PP_n))\lesssim \lb\frac{A_{n}}{\e}\rb^{2\rho_n},\quad t=1,2.
\end{align}
General \Holder\ classes satisfy \eqref{entropy bound: large class} \citep[cf. p.154,][]{wc}.

\begin{theorem}\label{theorem: GE: strong separation}
Suppose the function-class $\U_n$  is such  that 
 \eqref{entropy bound: large class} holds with $\rho_n\in(0,1)$, $A_n> 1$, where it also holds that  ${A_n^{2\rho_n}}/{n}\to 0$.
 Further suppose
  the approximation error
\[V_{\psi}^*-\sup_{(f_1,f_2)\in\U_n}V_{\psi}(f_1,f_2)=O((\log n)^k/{n})\]
for some $k\in\NN$, the optimization error $\Opn<1/2$, and Assumptions I-IV, \ref{assump: bound on Y}, and \ref{assump: boundedness} hold. Then there exist $C>0$ and $N_0>0$ such that for all  $n\geq N_0$ and for any $x>0$,
\[V_{\psi}^*-V_{\psi}(\hfa,\hfb)\leq  C\max\lbs (1+x)^{2/(1+\rho_n)} {A_n^{2\rho_n/(1+\rho_n)}}{n^{-1/(1+\rho_n)}},\ \Opn\rbs\]
with probability at least $1-\exp(-x)$.
\end{theorem}
Theorem \ref{theorem: GE: strong separation} is proved in Supplement~\ref{secpf: theorem: GE: strong separation}.
The proof of Theorem \ref{theorem: GE: strong separation} shows that under  \eqref{entropy bound: large class}, the estimation error is larger than the approximation error.   Thus, in this case, the biggest contribution in the regret comes from the estimation error  term, which is of the order $O_p\lb A_n^{2\rho_n/(1+\rho_n)}n^{-1/(1+\rho_n)}\rb$.
This type of rate is also observed for the $L_2$ risk in  regression problems where the regression function is a member of a class with similar entropy bounds \citep[p. 75, Example 4,][]{koltchinskii2009}. 
When $\rho_n\to 0$, the complexity of the class decreases, and as a result, the estimation error decreases as well. 

Our next theorem, which is proved in Section \ref{secpf: theorem: GE: strong GE: basis}, considers VC-type  $\U_n$'s similar to  Theorem~\ref{thm: weak: general thm}. The corresponding entropy bound is smaller than that in \eqref{entropy bound: large class}, which results in  smaller  regret bound compared to  Theorem \ref{theorem: GE: strong separation}. 
\begin{theorem}\label{theorem: GE: strong GE: basis}
Suppose $\U_n$  is a function-class such that  there exist $A_n>0$ and $\rho_n>0$ so that
\begin{equation}\label{intheorem: statement: GE: strong GE}
    N_{[\ ]}(\e,\U_n, L_2(\PP_n))\lesssim \lb\frac{A_{n}}{\e}\rb^{\rho_n},
\end{equation}
and 
$\liminf_n(\rho_n\log A_n)>0$.
 Further, suppose
  the approximation error
  \begin{equation}\label{ineq: regret}
      V_{\psi}^*-\sup_{(f_1,f_2)\in\U_n}V_{\psi}(f_1,f_2)=O((\log n)^k/{n})
  \end{equation}
for some $k\in\NN$. 
Then under Assumptions I-IV, \ref{assump: bound on Y}, and \ref{assump: boundedness},  there exist $C>0$ and $N_0>0$ such that for all  $n>N_0$ and any $x>0$,  
\[V_\psi^*-V_\psi(\hfa,\hfb)\leq C \max\lbs \frac{(1+x)^2(\log n)^2(\rho_n\log A_n)^2+(\log n)^k}{n},\ \Opn\rbs\]
 with $\PP$-probability at least $1-\exp(-x)$.
\end{theorem}

We will see that in our neural network example, $A_n$ is a polynomial in $n$, and   $\rho_n$ is a poly-log term, which leads to a regret of order $O_p(1/n)$  up to a polylog term. In binary classification context, many methods are known to attain this sharp rate \citep{massart2006, blanchard2008}. However, to the best of our knowledge, our DTRESLO method is the first DTR method for which such a sharp regret  bound is established under our nonparametric setup. 
Theorem S.1 of \cite{ zhao2015} implies that if $\G$ and $\Gt$ are uniformly bounded away from both zero and one, then  the regret of BOWL and SOWL estimators can be pushed to the order of $O_p(n^{-\omega/(1+\omega)})$ under Assumption~\ref{assump: boundedness}, where $\omega>0$ is a fixed quantity, called geometric noise exponent, that depends only on $\PP$. This rate is slower than ours when the optimization error $\Opn$ is of the order $O_p(1/n)$. Theorem 3.1 of \cite{qian2011} can be used to obtain regret bounds for Q-learning type methods under  Assumption~\ref{assump: boundedness} and \ref{assumption: smmothness}. However, the resulting regret bound has the same order as  the $L_2(\PP)$ estimation error of the Q-functions \citep[cf. equation 3.6 of][]{qian2011}, which can not be faster than $n^{-2\mybeta/(2\mybeta+p)}$  under nonparametric set-up \citep{yang1999minimax}.



 \subsection{Examples of regret bounds with specific $\U_n$'s}
\label{sec: examples of Un}

We will first state some assumptions on $\eta_1$ and $\eta_2$ that ensure the approximability of $\Gc$ and $\Gct$ by $\U_n$ when $\U_n$ corresponds to  basis-expansion type classes. Next we will elaborate on the special cases when $\U_n$ corresponds to neural networks and wavelets classes.

\subsubsection{Smoothness assumption}
First, we explain why smoothness conditions on $\eta_1$ and $\eta_2$ are required.  To control the  estimation error, 
 some structures on $\U_n$ are desirable because  smaller search spaces for the $f_t$'s result in smaller estimation error. Restricting the search space is equivalent to restricting the complexity of the class $\U_n$ \citep{audibert2007}. We require structural assumptions on $\G$ and $\Gt$ to ensure that $\G$ and $\Gt$ are well approximable by such  $\U_n$'s  -- giving rise to the so-called complexity assumptions. Thus, the complexity assumption enables the attainment of a small estimation error without necessarily blowing up the approximation error. 
See \cite{audibert2007}, \cite{koltchinskii2009} and \cite{tsybakov2004}, among others, for a more detailed account of the necessity of complexity assumptions. In the classification context,  $\U_n$ is taken to be some smoothness class or VC class  because most popular classifiers, e.g. neural network, basis-expansion type classifiers,  wavelets  etc. belong to such classes. Such classes can approximate $\G$ and $\Gt$ well if the latter are smooth. 
 Therefore, we will assume  our $\G$ and $\Gt$ belong to smoothness classes. To that end, we define the \Holder\  classes
 with smoothness index $\mybeta>0$ below.

Let $p\in\NN$. A  function $f:\mathcal \X\subset\RR^p\mapsto \RR$ is said to have H{\"o}lder smoothness index $\mybeta>0$ if for all $u=(u_1,\ldots,u_p)\in \NN^p$ satisfying $|u|_1<\mybeta$, $\partial^u f=\partial^{u_1}\partial^{u_2}\ldots\partial^{u_p} f$ exists and there exists a constant $C>0$ so that
\[\frac{|\partial^u f(x)-\partial^u f(y)|}{|x-y|^{\mybeta-\floor*{\mybeta}}}<C\quad \text{ for all }x,y\in \X.\]
For some $\Y>0$, we denote by $\C^\mybeta_d(\X,\Y)$ the H{\"o}lder  class of functions given by{\small
\begin{align}\label{def: holder}
\lbs f: \X\subset \RR^d\mapsto\RR\ \bl\  \sum_{u:|u|_1<\mybeta}\|\partial^u f\|_\infty+\sum_{u:|u|_1=\floor*{\mybeta}}\sup_{\substack{x,y\in \mathcal \X\\ 
x\neq y}}\frac{|\partial^u f(x)-\partial^u f(y)|}{|x-y|^{\mybeta-\floor*{\mybeta}}}\leq \Y\rbs.
\end{align}}

\noindent
Since  $H_t$ may include categorical variables such as smoking status, we separate the continuous and categorical parts of $H_t$ as $H_t=(H_{ts},\ H_{tc}) \in \H_t=\H_{ts}\otimes \H_{tc}$,  where $H_{ts} \in \H_{ts} \subset \RR^{p_{ts}}$ and $H_{tc}\in \H_{tc}\subset \RR^{p_{tc}}$ correspond to the continuous and categorical part of $H_t$, for $t=1,2$. 
\begin{assumption}[Smoothness assumption]
\label{assumption: smmothness}
  $\H_{1}$ and $\H_{2}$ are compact and $\H_{1c}$ and $\H_{2c}$ are finite sets. Also,  there exist $\mybeta>0$ and $\Y>0$ so that the followings hold:
\begin{itemize}
    \item[1.] Let $\X=\H_{1s}$. For each $h\in \H_{1c}$, the map $\G(\cdot,h):\X\mapsto\RR$ is in $ \C^\mybeta_{p_{1s}}(\X,\Y)$.
    \item[2.] Let $\X=\H_{2s}$. For each $(h,a)\in \H_{2c}\times\{\pm 1\}$, the function  $\Gt(\cdot,h,a):\X\mapsto\RR$  is in $ \C^\mybeta_{p_{2s}}(\X,\Y)$.
\end{itemize}
\end{assumption}
We formulated the smoothness assumption in terms of $\G$ and $\Gt$ so that our results are consistent with contemporary classification literature. However, our proofs show that one could formulate the assumptions in terms of the smoothness of the blip functions in \eqref{def: treatment effect stage 1} and \eqref{def: treatment effect stage 2} as well. 
The compact support assumption for $H_t$, which is typically satisfied in real applications, is also  commonly required in the DTR literature \citep{zhao2012, zhao2015, sonabendw2021semisupervised, zhang2018interpretable}. Under the compactness assumption, $\H_{1c}$ and $\H_{2c}$ are finite sets. Smoothness conditions as Assumption \ref{assumption: smmothness} have appeared in DTR literature in the context of nonparametric estimation  \citep{sun2021stochastic}.
Compared to the parametric assumptions often imposed on the blip functions in Q-learning or A-learning \citep{schulte2014}, our smoothness assumptions are much weaker. Our smoothness assumption includes non-differentiable functions as well. 
Next, we will establish regret bounds for neural network and wavelets classes.

\subsubsection{Neural networks as an example of $\U_n$ }
   \label{ex; neural network}
   
We consider the neural network space in line with \cite{schmidt2020}'s construction. Let ${\mathcal F}(L,\p,s,\Y)$ be the class of ReLU networks uniformly bounded by $\Y>0$, with depth  $L\in\mathbb N$,  width vector $\p$, sparsity $s\in\mathbb N$, and weights bounded by one. The output layer of the networks in ${\mathcal F}(L,\p,s,\Y)$ uses a linear gate. 
In this example,  we  consider that for $t=1,2$, the class $\U_{tn}$ corresponds to ${\mathcal F}(L_n,\p_n,s_n,\Y_n)$ where $L_n$, $\p_n$, $s_n$, and $\Y_n$ may depend on $n$. To avoid cumbersome notation, we drop $n$ from $L_n$, $\p_n$, $s_n$, and $\Y_n$, and simply denote them by $L$, $p$, $s$, and $\Y$, respectively. One can control the complexity of this class via pre-specifying upper bounds on the depth, width, and sparsity of the network.  We will first consider regret bound under Assumption~\ref{assumption: small noise}, and then we move to the special case of strong separation (Assumption \ref{assump: boundedness}).

 Corollary \ref{theorem: GE: neural network} establishes the regret bound of  DTRESLO  with neural network classifier under Assumption~\ref{assumption: small noise}.
 \begin{corollary}
 \label{theorem: GE: neural network}
 Suppose $\mathbb P$ satisfies Assumptions I-IV,  Assumption~\ref{assump: bound on Y}, Assumption~\ref{assumption: small noise} with parameter $\alpha>0$, and Assumption~\ref{assumption: smmothness} with parameter $\mybeta>0$.
Let $\mathcal U_{n,1}$ and $\mathcal U_{n,2}$ be of the form $\mathcal F(L,\p, s, \infty)$ with appropriate  $\p_1$, where  $\mathcal F(L,\p, s, \infty)$ is as defined in Section~\ref{ex; neural network}. Suppose $L=c_1\log n$,  $s=c_2 n^{p/((2+\alpha)\mybeta+p)}$, and the maximal width $\max W\leq c_3 s/L$ where $c_1,c_2,c_3>0$. Then there exist $N_0>0$ and $C>0$ depending on $\PP$ and $\psi$ such that if $c_1,c_2,c_3>C$,  then for  $n\geq N_0$ and any $x>0$, the following holds with probability at least $1-\exp(-x)$:
 \[V^*_\psi-V_\psi(\hfa,\hfb)\leq C\max\lbs (1+x)^2(\log n)^{\frac{6+4\alpha}{2+\alpha}}n^{-\frac{1+\alpha}{2+\alpha+p/\mybeta}},\ \Opn\rbs.\]
 \end{corollary}
 The proof of Corollary~\ref{theorem: GE: neural network} can be found in Section~\ref{sec: proof of gen: weak sep: NN}. The proof of Corollary  \ref{theorem: GE: neural network} assumes that $p$ is fixed, i.e. it does not grow with $n$. The generic constant $C$ in Corollary~\ref{theorem: GE: neural network} may depend on $p$ as well.
 Under Assumptions similar to \ref{assump: bound on Y}, \ref{assumption: small noise}, and \ref{assumption: smmothness}, the rate $n^{-\frac{1+\alpha}{2+\alpha+p/\mybeta}}$ is minimax in context of binary classification \citep{audibert2007}. Since  two-stage weighted classification problem is not easier than binary classification,  this rate is expected to be the minimax rate under our set-up as well. To the best of our knowledge, no other nonparametric DTR method has better  guarantees for the regret under set-up similar to ours. See Section~\ref{sec: related literature} for a  comparison of the regret bound of our DTRESLO method with some other existing  methods.

\begin{corollary}\label{cor: strong sep: NN}
Suppose $\mathbb P$ satisfies Assumptions I-IV,  Assumption~\ref{assump: bound on Y} with $\mathcal B>0$, Assumption~\ref{assump: boundedness}, and Assumption~\ref{assumption: smmothness} with parameter $\mybeta>0$.
Let $\mathcal U_{n,1}$ and $\mathcal U_{n,2}$ be of the form $\mathcal F(L,\p, s, \infty)$ with appropriate  $\p_1$, where  $\mathcal F(L,\p, s, \infty)$ is as defined in Section~\ref{ex; neural network}. Suppose $L=c_1\log n$,  $s=c_2 (\log n)^{p/\mybeta}$, and the maximal width $\max W\leq c_3 \log n$ where $c_1,c_2,c_3>0$. Then there exist $C,C'>0$ depending on $p$, $\mybeta$, and $\psi$ such that if $c_1,c_2,c_3>C$,  then   for  all $n\geq N_0$ and any $x>0$,  
\[ V_\psi^*-V_{\psi}(\hfa,\hfb)\leq  C'\max\lbs \frac{(1+x)^2(\log n)^{3+p/\mybeta}}{n},\ \Opn\rbs\]
 with probability at least $1-\exp(-x)$.
\end{corollary}
The proof of Corollary~\ref{cor: strong sep: NN} can be found in Section \ref{secpf: cor: strong sep: NN}. Corollary~\ref{cor: strong sep: NN} implies that DTRESLO method with neural network classifier can attain regret  $O_p(1/n)$ up to a poly-log term, as announced earlier.

\subsubsection{Wavelets as an example of $\U_n$}
\label{sec: wavelets: short}

In this section, we will define the wavelet estimators we will use.
Let us consider father and mother wavelets $\nu$ and $\zeta$. We will assume that these wavelets are S-regular wavelets \citep[cf. Definition 4.2.14, p.326, ][]{GineNickl} with $S>\mybeta$ and they have compact support, e.g.,  Daubechies wavelets \citep[p. 318,][]{GineNickl}. The corresponding $p$-dimensional wavelets
 are constructed using tensor products
 \[\Gamma({x})=\nu(x_1)\cdots\nu(x_p),\quad {x}=(x_1,\dots,x_p)\in K,\]
 where $K\subset\RR^p$ is a compact set. Define
  $\Gamma_{k}({x})=\Gamma({x}-{k})$ for $k\in\mathbb{Z}$, where $\mathbb Z$ was denoted to be the set of all integers. Observe that since $\nu$ has compact support,   $\nu(\cdot -k)$ restricted to any compact set is non-zero only for finitely many $k$'s. Therefore, $\Gamma_k$ restricted to compact set $K$ is non-zero only for finitely many $k$'s. We denote the set of such $k$'s by $\mathcal K_0$. For each $i\in \{0,1\}^p\setminus(0,\ldots 0)$, we also consider the $p$-dimensional tensor product 
  \[\mathcal{Z}^{i}({x})=\zeta^{i_1}(x_1)\cdots\zeta^{i_p}(x_p),\]
  where $\zeta$ is the mother wavelet function. 
   For $x\in K$, we define the wavelet function ${\mathcal{Z}}^{i}_{lk}({x})=2^{lr/2}{\mathcal{Z}}^{i}(2^l{x}-{k})$, where $l\in\mathbb{N}\cup\{0\}$, and $k\in\mathbb{Z}$. Notice that since $\zeta$ has compact support and $K$ is also compact, the function $x\mapsto{\mathcal{Z}}(2^l x-{k})$ is non-zero  only for some specific $k$'s belonging to a  finite set. We will call this set
  $\mathcal K(l)$. It is easy to show that there exist $C_1$ and $C_2>0$ so that the cardinality of $\mathcal K(l)$ satisfies $C_12^{lp}\leq |\mathcal K(l)|< C_2 2^{lp}$ for all $l\geq 1$.
  We let $l\in\{0,1,\ldots, b_n\}$, where $b_n$ is a sequence of integers diverging to $\infty$. This number $b_n$ will be called the level of the wavelet class. Finally, we define the wavelet function-class by
\begin{align}\label{wav: def: Hn}
\H^M_n(K) = \bigg\{ &\ 
f_n:K \mapsto\RR
\ \bl\  
f_n=\sum_{{k}\in\mathcal K_0}c_{k}\Gamma_{k}+\sum_{l=0}^{b_n}\sum_{{k}\in\mathcal K(l),{i}\in\mathcal{I}}c_{l{k}}
{\mathcal{Z}}_{l{k}}^{i} \text{ where }c_k\in\RR \nn\\ &\ \text{ for all }k\in\mathcal K_0\text{ and } c_{lk}\in\RR\text{ for all }l\in\mathbb{Z}\text{ and }k\in\mathcal K(l), \nn\\
&\ \sup_{k\in\mathcal K_0}c_k+\sup_{l\geq 0}\slbs 2^{l(\mybeta+p/2)}\sup_{k\in \mathcal{K}(l)}|c_{lk}|\srbs<M,\ \mathcal I=\{0,1\}^p\setminus (0,\ldots,0)\rbs.
\end{align}


The following two corollaries give the regret bound for the wavelet-based classifiers. As usual, we assume that $p\ll n$.
The following corollary, which is proved in Section~\ref{secpf: Proof of wavelets: weak convergence}, indicates that under the small noise condition, the regret decay rate is similar to that of the neural network example up to a poly-log term. Therefore, similar to the neural network example, the regret decay rate matches the minimax rate of risk decay in binary classification under  conditions similar to Assumptions~\ref{assump: bound on Y}, \ref{assumption: small noise}, and \ref{assumption: smmothness} \citep{audibert2007}.
\begin{corollary}
\label{cor: wavelets: weak convergence}
Suppose  $\{{\Gamma}_{k}\}_{k\in\mathcal K_0}$ and  $\{\mathcal{Z}_{lk}:{k\in\mathcal K(l), l\in\mathbb{N}\cup \{0\}}\}$ form an compactly supported S-regular wavelet basis  where $S>\mybeta>0$.
For $t=1,2$, let $\U_{tn}$ denote the  wavelet class  $\H_n^{M}(K_t)$ defined in \eqref{wav: def: Hn} with $M=2n$ and level $b_n=(\log_2n)/(2\mybeta+\alpha\mybeta+p)$, where $K_t$'s are  compact sets such that $K_1\supset \dom(\G)$ and $K_2\supset \dom(\Gt)$. If in addition Assumptions I-IV, Assumption \ref{assump: bound on Y} hold, Assumption \ref{assumption: small noise} holds with $\alpha>0$, and Assumption \ref{assumption: smmothness} holds with $\mybeta>0$, then there exist $C,N_0>0$ such that for all  $n\geq N_0$ and any $x>0$,
\[V_\psi^*-V_\psi(\hfa,\hfb)\leq C\max\lbs  (1+x)^2(p2^p)^{\frac{1+\alpha}{2+\alpha}}(\log n)^{\frac{5+3\alpha}{2+\alpha}}n^{-\frac{1+\alpha}{2+\alpha+p/\mybeta}},\ \Opn\rbs\]
with probability at least $1-\exp(-x)$.
\end{corollary}

The following corollary gives the regret decay rate under Assumption~\ref{assump: boundedness},  the strong separation assumption. This rate is similar to the analogous rate for the neural network example up to a poly-log term when $p\ll n$. Corollary \ref{lemma: strong separation: wavelets} is proved in Section~\ref{secpf: proof of corollary strong separation wavelets}.
\begin{corollary}
\label{lemma: strong separation: wavelets}
Suppose $\mathbb P$ and $\mathcal U_n$ are as in Corollary \ref{cor: wavelets: weak convergence}. Further, suppose  Assumptions I-IV, \ref{assump: bound on Y},  \ref{assump: boundedness}, and \ref{assumption: smmothness} hold with smoothness parameter $\mybeta>0$. Then there exist constants $C,C',N_0>0$ such that if $n\geq N_0$ and $b_n>C$, then for  any $x>0$,
\[V_\psi^*-V_\psi(\hfa,\hfb)\leq C'\max\lbs \frac{(1+x)^22^{2Cp}(\log n)^4}{n},\Opn\rbs\]
 with probability at least $1-\exp(-x)$.
\end{corollary}

{\color{black}
\section{Optimization error}
\label{sec: opt: summary}
 For the sake of brevity, we have moved the detailed discussion on optimization to Supplement \ref{sec:opt error}, and only summarize the key results in this section. In our analysis of optimization error, we exclusively focus on policies that are linear combinations of features. Specifically, we consider policies of the form $\tha(H_1)=J_1^T\theta$ and $\tha(H_2)=J_2^T\beta$, where $\theta\in\RR^{k_1}$ and $\beta\in\RR^{k_2}$ are the parameters, and $J_1 \in \mathbb{R}^{k_1}$ and $J_2 \in \mathbb{R}^{k_2}$ are functions of  $H_1$ and $H_2$ respectively, possibly involving basis functions or polynomials.  Even though $J_1$ and $J_2$ could depend on $n$, we omitted $n$ from their notation for the sake of simplicity. 
Furthermore, $V(\theta,\beta)$, $V_\psi(\theta,\beta)$, $\widehat V_\psi(\theta,\beta)$ will denote the value function, surrogate value function, and the empirical value function, respectively, of  the linear policies specified by $\theta$ and $\beta$.

  In addition to Assumptions I-IV, we assume that the conditional treatment effects or the blip functions defined in \eqref{def: treatment effect stage 1} and \eqref{def: treatment effect stage 2} are linear in $J_1$ and $J_2$, respectively. Specifically, we posit $\mathcal T_1(H_1)=\tk^TJ_1$ and $\mathcal T_2(H_2)=\bk^T J_2$, where $\tk\in\RR^{k_1}$ and $\bk\in\RR^{k_2}$ are unique constants.  Additionally, we assume that these treatment effects are non-zero with a probability of one.
  Furthermore, we assume that the features and the outcomes are bounded by a constant $C_{max}>0$.  Moreover, we consider $\phi$ functions that possess concave characteristics in the positive half of the X-axis and convex traits in the negative half. For example, the smooth $\phi$'s in Example \ref{Example: sigmoid} will satisfy this property. While some  results in Supplement \ref{sec:opt error} hold in more general setups than the one outlined above, we adopt these assumptions for the sake of ensuring a streamlined presentation within this section. Under the current setup, the optimal first stage treatment assignment is $1[\tk^TJ_1>0]$ and the optimal second stage treatment assignment is $1[\bk^TJ_2>0]$, implying $V^*=V(\tk,\bk)$.

\paragraph{Landscape analysis:} In the context of optimization problems, landscape analysis refers to the examination of the surface of the objective function. This analysis is particularly crucial in non-convex optimization problems, as it provides insights into the location of critical points and the global optimum.
   Lemma \ref{lemma: opt: positive data} in Supplement \ref{sec: absence of critical points} sheds light on the critical points of our objective, i.e., the empirical surrogate value function, in compact sets.  From a high level, this lemma says that this function has no critical point in compacta with probability tending to one.  Therefore, the supremum of our objective function is not attained in any compact set with high probability. A pivotal step in establishing this lemma is demonstrating that $V_\psi$, the   surrogate value function, has no critical point in  $\RR^{k_1+k_2}$. This characteristic of the surrogate value is  inherited from the underlying surrogate $\psi$, which attains maxima at $(\infty,\infty)$ and possesses no critical point in $\RR^2$.  Lemma \ref{lemma: opt: sup} in Supplement \ref{sec: opt: sup}  complements Lemma \ref{lemma: opt: positive data} by entailing that  the supremum of $V_\psi$ is  $C_\phi^2 V(\tk,\bk)$, and it is obtained as a limit in the sense that $V_\psi(a\tk,b\bk)$ approaches $C_\phi^2V(\tk,\bk)$ as $a\to\infty$ and $b\to\infty$. Furthermore, for all other $\theta\in\RR^{k_1}$ and $\beta\in\RR^{k_2}$, it turns out that $V_\psi(a\theta,b\beta)$ approaches sub-optimal values as $a\to\infty$ and $b\to\infty$. In the special case where $\theta\in\RR$ and $\beta\in\RR$, the above implies that $V_\psi$ is maximized at one of the following extended-valued-tuples:  $(\infty,\infty)$, $(-\infty,\infty)$, $(\infty,-\infty)$, and $(-\infty,-\infty)$.  
 
  Lemma \ref{lemma: opt: diff} in Supplement \ref{sec: opt: sup} shows that for 
  any $(\theta,\beta)$, the $\psi$-regret $V_\psi(\theta,\beta)$ is closely approximated by 
a scaled version of its actual regret when the $l_2$ norms of $\theta$ and $\beta$ are sufficiently large. Using this result, we demonstrate that if  $V_\psi(\theta,\beta)$ is close to the supremum of $V_\psi$, $V(\theta,\beta)$ is also close to $V^*$, suggesting an analog of Fisher consistency for our surrogates within the restricted class of linear policies. This result ensures that any policy with a small $\psi$-regret  will indeed be a high-quality policy under our setup.    Moreover, when the value function is continuous in $\theta$ and $\beta$, our heuristic analysis indicates that $V_\psi(\theta,\beta)$ becomes arbitrarily close to the supremum when the angle between $(\theta,\beta)$ and $(\tk,\bk)$ is small and the $l_2$ norms of $\theta$ and $\beta$ are sufficiently large. In addition, since $V_\psi$ is bounded, we conjecture that $V_\psi$  has superlevel sets where it is concave. 
We anticipate that in large samples, the empirical value function $\widehat V_\psi$ display similar properties, and exhibits concavity in some superlevel sets.  We provide a toy example in Supplement \ref{sec: toy example} (see  Figure \ref{fig:loss surface} in the Supplement), where the above holds true. 

\paragraph{Convergence analysis:}
Result \ref{result: gradient} in Supplement \ref{sec: opt: convg} shows that our objective function has  globally Lipschitz gradients. Using this result alongside classical results on gradient descent and stochastic gradient descent \citep{bottou2018optimization}, we argue that the gradient descent and stochastic gradient descent iterates  will diverge under the above-mentioned setup. This is unsurprising given the absence of critical points in $\RR^{k_1+k_2}$. Also,  for properly chosen stepsizes, our objective will strictly decrease at each iteration  in this case (gradient descent is applied to the  minimization problem). If the gradient descent or stochastic gradient descent iterates enter the  superlevel sets mentioned in the previous paragraph,   the empirical value function $\widehat V_\psi$ will converge to the supremum, provided the conjectured concavity holds on these sets. In that case, the optimization error $\Opn$ will approach zero. 
However,  the gradient descent iterates can also diverge in other directions, leading to sub-optimal values and non-negligible optimization error. It appears to us that the initiation probably influences where the iterates will diverge. Lastly, while the iterates may theoretically diverge when gradient-descent type algorithms are applied to the objective $-\widehat V_\psi$,  in practice, the algorithms will stop after finitely many iterations, provided the stopping criteria is based on the change in the objective function. This is because this objective is bounded below, and it strictly decreases at each iteration. 

 Finally, we show that  Polyak- \L{}ojasiewicz (PL) Inequality and other well-known conditions that ensure convergence to the global maximum in general non-convex problems do not apply to our optimization scenario. Furthermore, while we expect our objective function to exhibit concavity in some superlevel sets, it will probably not exhibit the strong concavity required for our desired rate guarantees. As a result,  using  existing results on gradient and stochastic gradient descent, we can only show that the squared $l_2$-norm of the gradients decays at a linear rate. Unfortunately, this limited information doesn't provide substantial insight into the optimization error. Obtaining specific rate-related results in the absence of PL inequality and strong concavity is generally challenging in non-convex problems \citep{patel2022gradient,bottou2018optimization,Patel2021GlobalCA,karimi2016linear}. Such an analysis would likely necessitate a comprehensive examination of gradient descent or stochastic gradient descent tailored to the specific structures inherent to our problem. 
 However, conducting a comprehensive analysis of this nature is beyond the current scope of our paper, and we view it as a potential avenue for future research.}
\section{Connection to related literature}
\label{sec: related literature}
  
  \subsection*{Comparison with SOWL \citep{zhao2015}}
  As mentioned earlier, our work is inspired by the SOWL method of \cite{zhao2015}, who uses the hinge loss $\psi(x,y)=\min(x,y,1)$ as the surrogate. Since hinge loss is non-smooth, the resulting surrogate value function is not amenable to gradient-based optimization methods.  The surrogate value function is therefore optimized via the dual formulation, which leads to SVM-type methods.  However, SVM-type approaches have several limitations. First, the number of linear constraints of the SVM formulation  scales linearly with the sample size, which leads to  the computational complexity of $O(n^3)$ \citep{NIPS2000_19de10ad}. This makes SVM-based methods not scalable for large-scale EHR studies where $n$ is large.  Second, the dual objective does not facilitate straightforward variable selection via the addition of an  $l_1$ penalty. Finally, SVM-type methods are not flexible enough to accommodate tree-based or neural-network-type classifiers with rigorous statistical guarantees. Our surrogate loss-based approach is flexible, and it has the advantage of both computational efficiency and the ability to incorporate variable selection. 
  
  \subsection*{Comparison of regret upper bound} 
Table~\ref{tab: rate compare} compares our regret bound with neural network and wavelets classifiers with the available regret bound for the following nonparametric DTR methods: BOWL/SOWL \citep{zhao2015}, nonparametric Q-learning, the list-based method of \cite{zhang2018interpretable}, and the stochastic tree-based reinforcement learning (ST-RL) method of \cite{sun2021stochastic}. The last method uses a Bayesian additive regression tree (BART). An important fact about the last two methods is that they especially target enhanced interpretability. We also remark that the upper bound on Q-learning regret under small noise assumption is derived using \cite{qian2011}'s results, which were for stage one.
We refer to Supplement~\ref{sec: on the table} for more  details on these methods and the construction of Table~\ref{tab: rate compare} using the results in the associated papers. In Table~\ref{tab: rate compare}, by ``noise assumption", we indicate assumptions on the data distribution at the treatment boundary; Assumption~\ref{assumption: small noise} is an example, but see  Section~\ref{sec: small noise} for more details. By ``smoothness assumption", we refer to assumptions on the smoothness of functions such as $\eta_1$, $\eta_2$,  the blip functions in \eqref{def: treatment effect stage 1} and \eqref{def: treatment effect stage 2}, or their derivatives.

It is important to note that the regret bounds of \cite{zhao2015} are originally produced under the geometric noise condition of \cite{steinwart2007}.  We use the correspondence between the latter condition and our Assumptions \ref{assumption: small noise} and \ref{assumption: smmothness}, which follows from \cite{steinwart2007} itself. See Supplement~\ref{sec: on the table} for more details on this.  Also, the interpretable methods and Q-learning do not consider the approximation error, and thus their regrets are with respect to the best treatment regime within the  classes under consideration. 
  \begin{table}[H]
      \centering
      
      \begin{tabular}{llll}
      Method     & \makecell[vl]{Regret bound\\ exponent} & \makecell[vl]{Noise \\assumption\\type} & \makecell[vl]{Smoothness\\ assumption\\type} \\
      \toprule\\
       \makecell[vl]{ DTRESLO with NN \\ and wavelets}   & $\min(\frac{1+\alpha}{2+\alpha+p/\mybeta},-\log_n\Opn)$  & Small  noise & \Holder\ ($\mybeta$)\\
       BOWL/ SOWL & $\begin{matrix}
 \frac{1}{2}\frac{1+\alpha}{1+\alpha+p/\mybeta}-\e^* &\text{if }\alpha\geq p/\mybeta-1\\
 \frac{1+\alpha}{3+3\alpha+p/\mybeta}-\e^* & \text{o.w.} 
 \end{matrix}$
 & Small noise & \Holder\ ($\mybeta$)\\
 \makecell[vl]{Nonparametric\\Q-learning$^{\dag}$}  & $\frac{(1+\alpha)}{2+\alpha}\frac{2}{2+p/\mybeta}$  & Small noise & \Holder\ ($\mybeta$)\\
 ST-RL & $\frac{1}{3}\frac{2}{2+p/\mybeta}$ & $\text{Different}^{**}$ & \Holder\ ($\mybeta$)\\
List- based & $\slb\frac{2}{3}\srb^{l^{\dag\dag}}\frac{1}{2+p/\mybeta}$ & $\text{Different}^{**}$ & \makecell{Modulus of\\ smoothness $(\mybeta)^{**}$}\\
\bottomrule 
      \end{tabular}
      \caption{Best possible regret-decay rate  given by available upper bounds on different DTRs: here, the decay rate is $O(n^{-\text{exponent}})$ up to a polylog term. Also, $\Opn$ is the optimization error mentioned in Section~\ref{sec: error decomposition}. Here NN: neural network.
     \begin{flushleft}
      $\vphantom{h}^*$Here $\e$ is any positive number.\\
        $\vphantom{h}^{**}$ See Supplement~\ref{sec: on the table} for more details.\\
        $\vphantom{h}^{\dag}$ This regret bound holds for one stage case.\\
         $\vphantom{h}^{\dag\dag}$  $l$ is the length of the longest classification list.
     \end{flushleft}
       }
      \label{tab: rate compare}
  \end{table}
 
  First note that, since $\alpha>0$, we always have  $1/3<(1+\alpha)/(2+\alpha)$. Thus Table~\ref{tab: rate compare} implies ST-RL has a  larger regret  bound  than  nonparametric Q-learning. However, ST-RL is not far from a   Bayesian version of nonparametric Q-learning. The list-based method of \cite{zhang2018interpretable} also faces an apparent loss in efficiency.  The less efficiency for these two methods is best perceived as the cost of enhanced interpretability  \citep{sun2021stochastic}. Anyway, since  no lower bound on these regrets is available, nothing can be said for certain. 
  
  By elementary algebra, it can be shown that when $\Opn$ is negligible, our rate is sharper than the rate guarantees of all  other methods in Table~\ref{tab: rate compare}. For the sake of simplicity, in the rest of the discussion, we assume that $\Opn$ is negligible.  We have already mentioned that \cite{zhao2015}'s original analysis uses the geometric noise assumption. In Supplement~\ref{sec: on the table}, we show that the geometric noise assumption is satisfied when MA-type small noise assumption  and Assumption~\ref{assumption: smmothness} both hold. From \cite{steinwart2007}, it follows that  under our  assumptions, binary classification with hinge loss almost attains the regret decay rate $n^{-(1+\alpha)/(2+\alpha+p/\mybeta)}$. (We use the word ``almost" here because the term $\e$ in Table~\ref{tab: rate compare} still can not be removed.)
  Had the small noise condition been used, exploiting the direct correspondence between BOWL and binary classification, it may be possible to obtain a similar upper bound for BOWL's regret. Therefore it is possible that BOWL also attains the regret decay rate of $n^{-(1+\alpha)/(2+\alpha+p/\mybeta)-\e}$   under our  assumptions.
  
 We speculate that improving the Q-learning regret bound
     may not be  easy under our current set of assumptions. The reason is that we anticipate  the term $2/(2+p/\mybeta)$ in the upper bound  probably can not be improved. This term stems from the lower bound on $L_2(\PP)$ error rates of nonparametric methods; see Supplement~\ref{sec: on the table} for more details. Therefore, we speculate that under our Tsybakov small noise assumption and our smoothness assumption, there may be a gap between the regret of NP Q-learning type methods and direct search methods such as ours or BOWL. To rigorously prove this claim, one needs to calculate the lower bound on the regret of two stage NP Q-learning, which is beyond the limits of the current paper. It must be kept in mind that, even if DTRESLO exhibits a sharper rate of regret decay than Q-learning under Tsybakov's small noise assumption,
     it does not inherently imply that DTRESLO is always theoretically superior to Q-learning. In the context of binary classification, \cite{audibert2007} noted that the theoretical comparison between model-based and model-free methods heavily relies on the assumptions imposed on the data distribution.  We anticipate that a similar notion applies when comparing regression-based methods and direct search methods in DTR. Stronger assumptions on the data distribution, such as the strong density assumption in \cite{audibert2004}, can facilitate easier nonparametric estimation, leading to a faster rate of regret decay for nonparametric model-based methods under the small noise assumption \citep{luedtke2016, qian2011}. {\color{black} \label{page: propensity scores} Furthermore, it is important to note that all our comparisons are based on the assumption that the propensity scores are known. When they are unknown, the regret decay rate of DTRESLO and BOWL will be slower, even under Tsybakov's small noise assumption. Therefore the comparison between direct search methods and Q-learning will differ under these circumstances.}

     There is also a possibility that if we opt for more interpretable classifiers, the theoretical efficiency will decline as in the case of  ST-RL and the classification list-based method of \cite{zhang2018interpretable}. Hence, there is a possibility of a trade-off among theoretical efficiency, computational hardness, and interpretability, at least in principle. Future research in this direction may be able to shed more light on this trade-off.
     


\section{Empirical Analysis}\label{section: simulation & application}
We have performed extensive simulations to evaluate the finite sample properties of our DTRESLO methods and to compare them with existing methods under different generating data mechanisms. We additionally evaluate the performance of the proposed method using EHR data to identify optimal treatment rules for patients in intensive care units (ICU) with sepsis.

\subsection{Simulations}
\label{sec: simulation}

We compare the performance of our DTRESLO method with the regression-based method Q-learning and the direct search methods BOWL and SOWL \citep{zhao2015}.
For our DTRESLO method, we take $\phi(x)=1+2/\pi \cdot \arctan(\pi x/2)$ because simulation shows that it has slightly better performance than the other smooth surrogates considered in Example~\ref{Example: sigmoid}. We consider several choices for the class of classifiers $\U_{1n}$ and $\U_{2n}$. When we consider the linear treatment policies,  $\U_{1n}$ and $\U_{2n}$ are the class of all linear functions on $\H_1$ and $\H_2$, respectively. We consider
 cubic spline, wavelets, and neural network (NN)  as the non-linear treatment policies, with  $\U_{1n}$ and $\U_{2n}$ being the respective function-classes in these cases. For the comparators, i.e. Q-learning, BOWL, and SOWL, we consider both linear  and non-linear policies  as well. Following \cite{zhao2015}, we incorporate non-linear policies  for  BOWL and SOWL using a reproducing kernel Hilbert space (RKHS) with  RBF kernel; see \cite{zhao2015} for more  details. The non-linear treatment policy for Q-learning is achieved by letting the  Q-functions be in neural network classes.  See Supplement~\ref{sec: details of simulation} for more details on the implementation of these methods.
 
We considered five broad simulation settings as detailed in Section~\ref{sec: details of simulation} of the Supplement:
\begin{enumerate}
    \item All covariates are discrete. Hence, an exhaustive search over $\F$ is possible using saturated models. 
    \item This is a setting with non-linear decision boundaries in both stages. However, $Y_2$ does not depend on $A_1$.
    \item This setting is inspired by Setting 2 of  \cite{zhao2015}, where the outcome models i.e. $\E[Y_t\mid H_t]$'s are  linear function of $H_t$ for $t=1,2$. We will call this setting the linear setting.
    \item This has highly non-linear and even non-smooth decision boundaries.
    \item This setting has a higher number of  covariates. In this case, the first stage outcome model is linear, but the second stage outcome model is non-linear.
\end{enumerate}
 Setting 1 is a simple toy setting. The motivation behind including this setting is to verify the consistency of DTRESLO. 
We will use the linear setting 3 
 to check if linear treatment policies perform well when the outcome models are linear.
 On the other hand, we include settings 2 and 4 to examine if the methods with non-linear treatment policies have an edge over those with linear treatment policies when the decision boundaries are non-linear. Finally, setting 5 is included to compare the performance of different methods when the dimension of $\O_1\cup \O_2$ is comparatively larger. 

Under each listed setting, we estimate the DTRs based on samples of size $n=250,2500,5000$.   For each estimated DTR $\widehat d$, we estimate  the value function $V(\widehat d_1,\widehat d_2)$   by the  empirical value function based on an independent sample of size 10,000. We estimate the expectation and the standard deviation of these value function estimates using 500  Monte Carlo replications. 
 We also estimate the optimal value function $V^*=V(d_1^*(H_1),d_2^*(H_2))$ for each setting using these 500 Monte Carlo replications. Figures~\ref{Fig: value: non-linear} and \ref{Fig: value: linear} compare the estimated expected value functions of the different methods under consideration. In these figures, we use the neural network DTRESLO as the non-linear DTRESLO because this method is comparable to neural network Q-learning. The average value functions corresponding to  the other non-linear DTRESLO methods can be found in  Table~\ref{Table:est values} in Supplement~\ref{sec: details of simulation}. The overall performance of all the non-linear DTRESLO methods is quite similar, although NN DTRESLO is slightly better than the rest.
 
 \begin{figure}[h]
 \centering
 \begin{subfigure}[b]{\textwidth}
     \includegraphics[height=2 in, width=\textwidth]{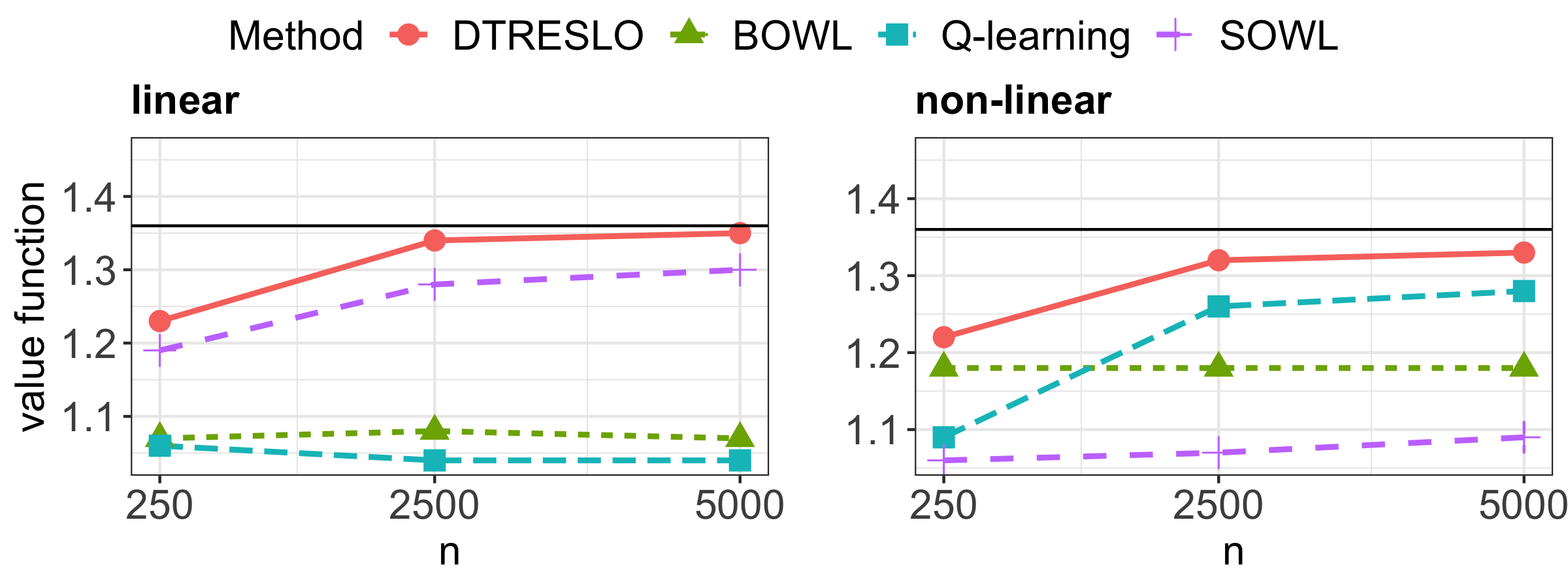}
 \caption{Setting 1.}
 \label{Fig: setting 1}
 \end{subfigure}
  \begin{subfigure}[b]{\textwidth}
     \includegraphics[height=2 in, width=\textwidth]{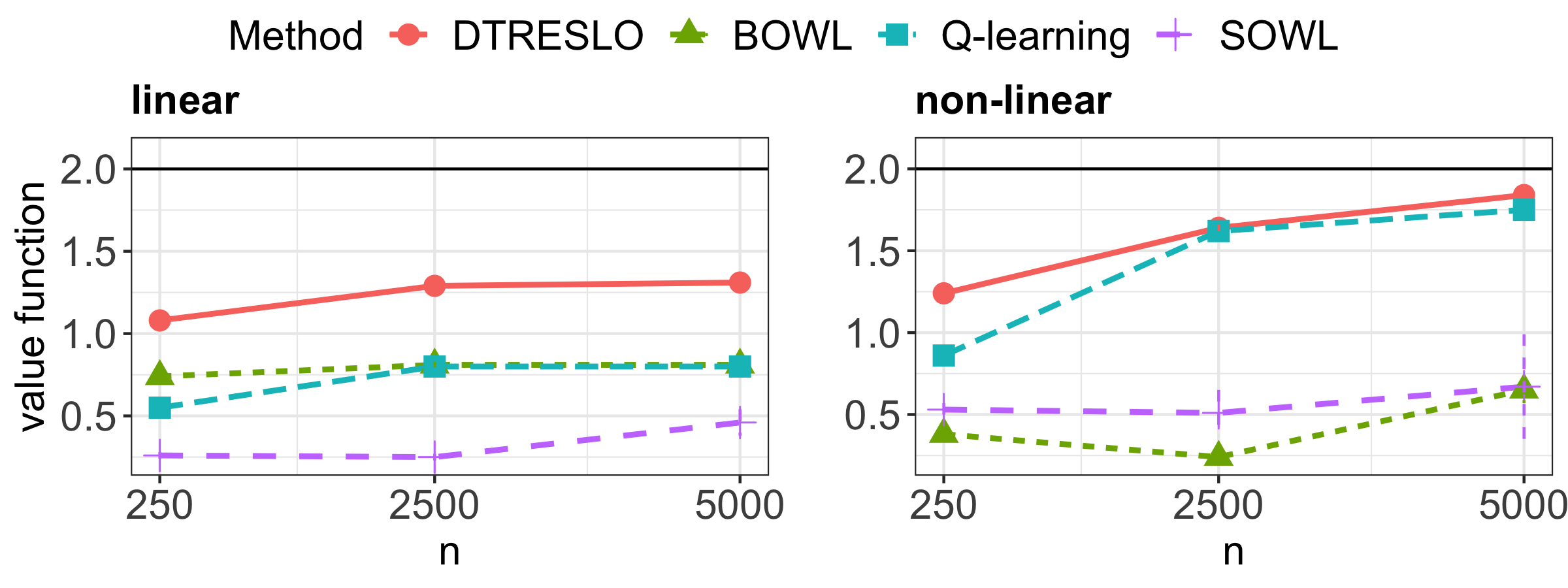}
 \caption{Setting 2.}
 \label{Fig: setting 2}
 \end{subfigure}\\
 \begin{subfigure}[b]{\textwidth}
     \includegraphics[height=2 in, width=\textwidth]{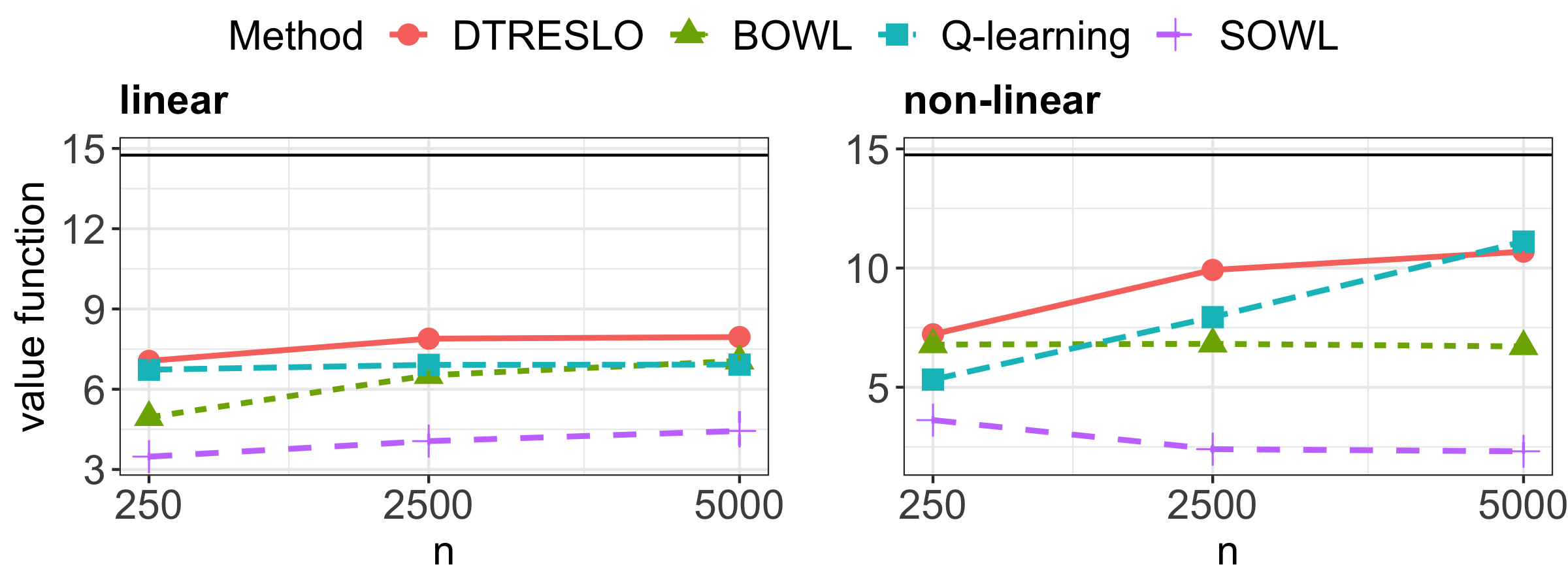}
 \caption{Setting 4.}
 \label{Fig: setting 4}
 \end{subfigure}
 \caption{Plot of the estimated average value functions for the settings 1, 2 and 4. Here the black horizontal line corresponds to the true value function. The left and right panels correspond to the linear and non-linear treatment policies, respectively. Here the non-linear DTRESLO corresponds to the neural network classifier. The error bars are given by $\pm 2$ SD.  }\label{Fig: value: non-linear}
 \end{figure}

\begin{figure}[h]
\centering
 \begin{subfigure}[b]{\textwidth}
     \includegraphics[height=2 in, width=\linewidth, keepaspectratio]{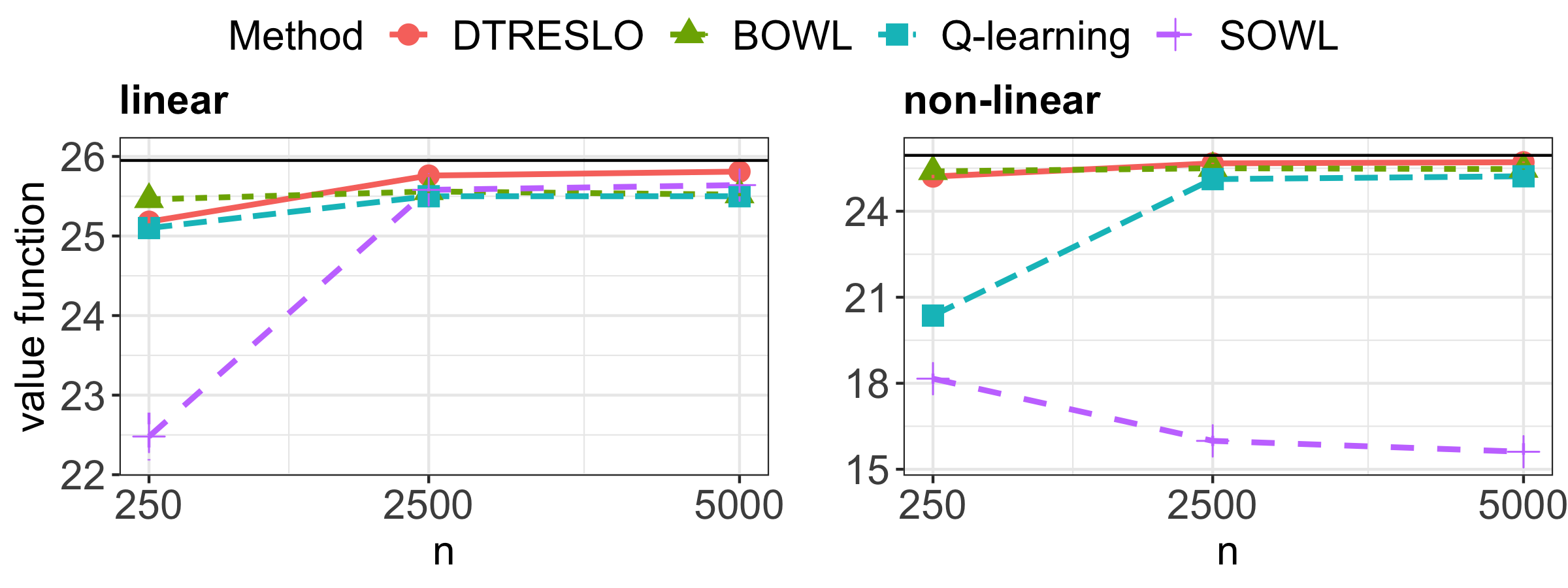}
 \caption{Setting 3.}
 \label{Fig: setting 3}
 \end{subfigure}
 \begin{subfigure}[b]{\textwidth}
     \includegraphics[height=2 in, width=\textwidth]{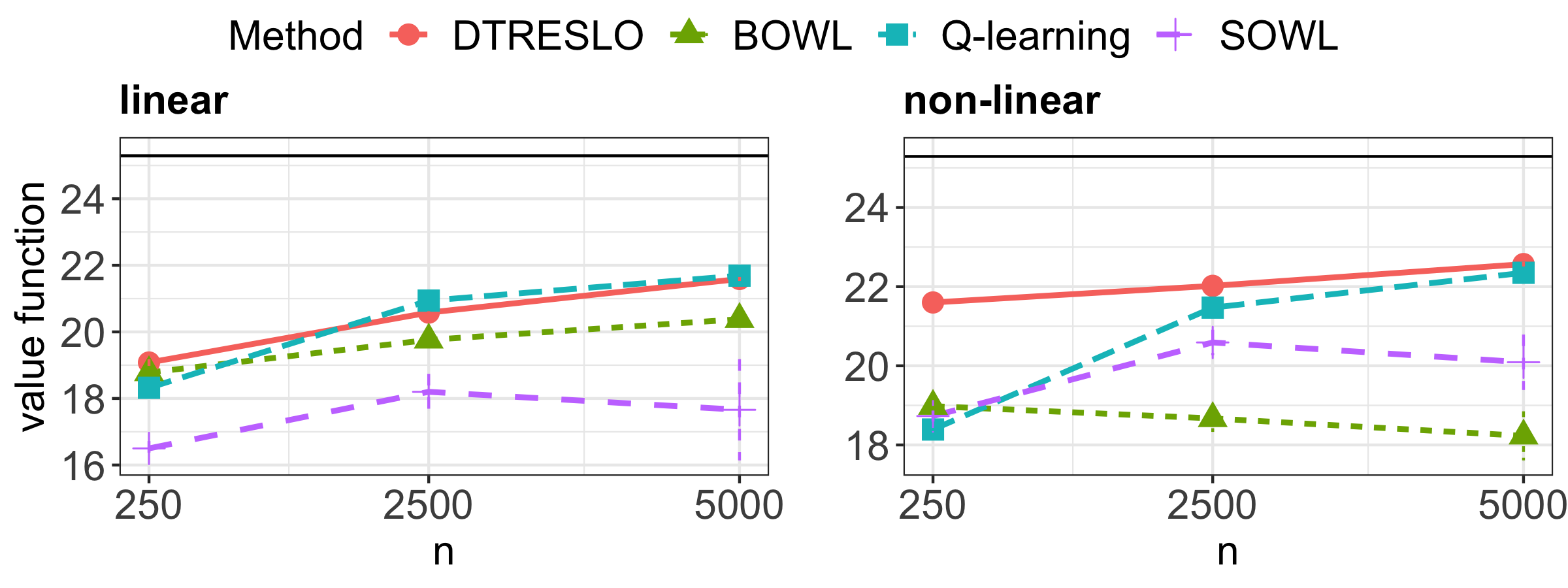}
 \caption{Setting 5.}
 \label{Fig: setting 5}
 \end{subfigure}
 \caption{Plot of the estimated average  value functions for the settings 3 and 5. Here the black horizontal line corresponds to the true value function. The left and right panels correspond to the linear and non-linear treatment policies, respectively. Here the non-linear DTRESLO corresponds to the neural network classifier. The error bars  are given by $\pm 2$ SD.  }\label{Fig: value: linear}
 \end{figure}


First of all, Figures~\ref{Fig: value: non-linear} and \ref{Fig: value: linear} entail that   DTRESLO consistently performs better or at least as good as  the other methods under all our settings and all sample sizes. No other method has reliable performance across all settings.  
First, we will  investigate the five settings in more detail. Then we will look more closely into the comparison between  DTRESLO and the other methods. Finally, we will compare the run-time of different methods.

 Figure~\ref{Fig: setting 1} underscores that in the simple setting 1, DTRESLO outperforms all other methods under both linear and non-linear treatment policies. 
 Figures \ref{Fig: setting 2} and \ref{Fig: setting 4} show that in the non-linear settings 2 and 4, as expected, the non-linear versions of DTRESLO, BOWL, and Q-learning perform better than the linear counterparts. The only exception is the case of SOWL, which we will discuss later in more detail. We also observe that setting 4 is quite hard in that the expected value function of all methods is noticeably lower than the optimal value function. In settings 2 and 4, non-linear DTRESLO performs noticeably better than non-linear Q-learning in a small sample ($n=$ 250).
  As the sample size increases, the difference decreases. SOWL has poor performance under both settings. Although BOWL has better performance than SOWL, its performance improves rather slowly with $n$ when compared to DTRESLO. 
  This difference is most noticeable for the non-linear treatment policies under large samples.

Figure~\ref{Fig: setting 3} implies that under the linear setting 3,  value function estimates of the linear treatment policies are as large as  the non-linear policies for all methods except SOWL. Setting 5, which has a larger number of variables, is a relatively more complicated setting.  Although the second stage outcome models are  non-linear in this setting, Figure~\ref{Fig: setting 5} underscores  that  linear DTRESLO performs quite  comparably to non-linear DTRESLO in large samples under this setting. Table~\ref{Table:est values} in Supplement \ref{sec: details of simulation} implies that the situation with the other non-linear DTRESLO methods is similar. Similar to settings 2, in this case, non-linear DTRESLO has a noticeable edge over all other methods when the sample size is 250.

Under all settings, non-linear DTRESLO and Q-learning exhibit one particular pattern, which merits some discussion. Non-linear DTRESLO performs better than non-linear Q learning in small samples, but their performance becomes almost similar when the sample size increases to 5000. The relative underperformance of nonparametric Q-learning in small samples may be due to its heavy  reliance  on the correct estimation of  Q-functions. Nonparametric estimation of functions is harder unless the sample size is sufficiently large. In contrast, DTRESLO only needs to estimate  the sign of the blip functions, which is easier than the estimation of the whole function. Finally, this difference may be the manifestation of the speculated faster regret decay of neural network DTRESLO (see  Section~\ref{sec: related literature}). Thus our simulation study complements the theoretical comparison of the regrets between nonparametric Q-learning and DTRESLO.

DTRESLO outperforms the other direct search methods, BOWL and SOWL, under all settings except the linear setting, i.e. setting 3, where BOWL and DTRESLO have  comparable performance.  The difference is most pronounced for non-linear treatment policies in large samples.
  DTRESLO's advantage over BOWL may be attributed to DTRESLO's simultaneous optimization approach as opposed to BOWL's stagewise approach. The  latter reduces the effective sample size in the first stage.  In general, SOWL's average value function  stays quite below the optimal value function. 
  Its performance is comparable to other methods only in  setting 3, where classification is comparatively easy. 

{\color{black} \label{page: setting 3}   Figures~\ref{Fig: value: non-linear} and \ref{Fig: value: linear} entail that the estimated value function of  non-linear SOWL, a nonparametric method by design,  either does not improve with the sample size or exhibits a much slower increase compared to the other competing methods we consider.  The last observation raises the question of whether the approximation error of SOWL at all decays to zero as the sample size increases. Indeed, this observation does not refute our Theorem~\ref{theorem: hinge}, which establishes that the hinge loss, the surrogate employed in SOWL, requires the fulfilment of  \eqref{inlemma: hinge solution first stage requirement} for $\tilde{d}_1(H_1)$ to align with $d^*_1(H_1)$. Moreover, in Supplement \ref{sup: hinge}, we demonstrate that \eqref{inlemma: hinge solution first stage requirement} is not a pathological condition, as it fails in numerous non-trivial scenarios. To elaborate further, we focus on Setting 3 as an illustrative case.  In this case, the non-parametric version of SOWL exhibits a decaying value with respect to $n$. For this setting, the outcome models are linear, and $H_1\in \RR^3$ follows a centered multivariate Gaussian distribution with an identity covariance matrix. Consequently, $H_1$ lies inside a ball of radius $5$ centered at the origin with a high probability (specifically, greater than $0.999$). However, we empirically evaluated that \eqref{inlemma: hinge solution first stage requirement} holds nowhere inside this ball. Moreover, if a location transformation is required to ensure the positivity of outcomes for certain samples, as discussed in Section \ref{sec: hinge loss}, \eqref{inlemma: hinge solution first stage requirement} becomes more difficult to satisfy. Therefore, the suboptimal performance of non-linear SOWL  may be attributable to the potential failure of \eqref{inlemma: hinge solution first stage requirement} in this case.}

\setlength{\tabcolsep}{2.5pt}
\begin{table}[h]
\centering
\resizebox{\columnwidth}{!}{
\begin{tabular}{cc cccc|cc|cc|cc}
\hline
\multirow{2}{*}{Setting}   & \multirow{2}{*}{$n$} & \multicolumn{4}{c|}{DTRESLO} & \multicolumn{2}{c|}{BOWL} & \multicolumn{2}{c|}{SOWL} &  \multicolumn{2}{c}{Q-Learning}
  \\ \cline{3-12}
  && Linear & Wavelet & Spline & NeuNet & Linear  & RBF & Linear & RBF  &  Linear & NeuNet               \\ \hline
\multirow{3}{*}{1} & 250 &	0.04 &	0.05 &	0.04 &	0.1 &	1.24 &	21.01 &	0.1 &	0.16 &	0.07 &	0.18 \\
 &2500 &	0.42 &	0.49 &	0.43 &	0.94 &	13.11 &	655.43 &	54.19 &	80.82 &	0.69 &	2.12 \\
 &5000 &	0.89 &	1.02 &	0.8 &	2.15 &	77.45 &	3913.85 &	400.32 &	534.36 &	1.4 &	3.94   \\ \cline{2-12} 
\multirow{3}{*}{2}&	250&		0.04&		0.06&		0.04&		0.16&		1.36&		3.48&		1.3&		1.33&		0.09&		0.19 \\
&	2500&		0.54&		0.6&		0.42&		1.01&		27.75&		271.08&		773.79&		822.42&		0.73&		1.83 \\
&	5000&		0.88&		1.25&		0.91&		3.7&		136.88&		5139.53&		5901.54&		5755.75&		1.49&		4.15 \\ \cline{2-12} 
\multirow{3}{*}{3}& 250&	0.05&	0.05&	0.04&	0.15&	12.59&	24.39&	0.08&	0.13&	0.08&	0.2 \\
&2500&	0.41&	0.71&	0.42&	1.03&	25.75&	704.16&	46.55&	106.67&	0.73&	2.04 \\
&5000&	1.22&	1.04&	0.84&	2.19&	107.99&	4063.34&	345.83&	859.37&	1.46&	5.36 \\ \cline{2-12} 
\multirow{3}{*}{4}&250	&0.04&	0.06&	0.04&	0.1&	1.66&	3.21&	1.28&	1.33&	0.09&	0.19\\ 
&	2500&		0.59&		0.49&		0.42&	1.04&	20.12&	424.81&	806.54&	833.64&	0.73&	1.88\\ 
&5000&	0.86&	1.32&	0.91&	1.5&	70.86&	3317.64&	5674.96&	5778.79&	1.47&	3.61 \\ \cline{2-12} 
\multirow{3}{*}{5}&250&	0.04&	0.05&	0.04&	0.09&	10.97&	18.05&	1.26&	1.32&	0.07&	0.18 \\
&2500&	0.6&	0.48&	0.42&	1.53&	33.46&	222.42&	810.59&	833.1&	0.72&	1.92 \\
&5000&	1.19&	1.26&	1.21&	2.1&	169.31&	1012.86&	5645.25&	6010.88&	1.38&	3.79 \\ \hline
\end{tabular}}
\caption{Run-time for estimating DTR for our smooth surrogates (DTRESLO), \cite{zhao2015}'s BOWL \& SOWL, and $Q$-learning under settings 1--5.}\label{Table:run time}
\end{table}

 Table \ref{Table:run time} tabulates the run-time of the DTR estimation methods. Run times for DTRESLO with linear, wavelets, and spline-based treatment policies are relatively similar. The run-time doubles for neural network treatment policies. Nonetheless, they are all less than three seconds. Both linear and neural network $Q$-learning methods  are slightly slower than their DTRESLO counterparts, but the difference in run-time is negligible. This is not surprising because DTRESLO and $Q$-learning methods are trained in a similar way. They all use stochastic gradient descent with RMSprop for optimization of the respective loss functions. All these methods are trained for 20 epochs and use a batch size of 128.   As expected, BOWL and SOWL  have a much larger run-time, which also increases sharply with $n$. This larger run-time is expected because SVMs utilize the dual space for optimization. The time cost is especially high in settings 2 and 4, which have highly non-linear decision boundaries, and setting 5, which has over 32 features.

To summarize, DTRESLO improves the scalability of existing direct search methods, achieving run-time as small as  Q-learning.
We also observe that within the same class of treatment regimes, i.e. linear or neural network, DTRESLO  outperforms regression-based Q-learning in small samples.  This observation aligns with the existing claim in the literature that  classification is easier than  regression in the context of DTR especially in small samples  \citep{zhao2015, laber2019}. This may happen because regression-based methods focus on minimizing the $L_2(\PP)$ loss, where  the estimation of optimal rules only requires minimization of the zero-one loss. This mismatch of loss has previously been discussed in literature \citep{murphy2005,qian2011}. Our observation thus hints that bypassing regression may result in  better-quality treatment regimes, at least in small samples.

\subsection{EHR Data Application: DTR for ICU Patients with sepsis}
\label{sec: application}

We evaluate our DTRESLO methods and benchmarks on a cohort of $n=9,872$ ICU patients with sepsis from the Medical Information Mart for Intensive Care version IV (MIMIC-IV)  data \citep{mimicIV}. Sepsis is a situation when body's response to an infection  overwhelms the body's immune system, potentially causing damage to tissue, multi-organ failure, and in some cases, death. Individualization of treatment strategies is highly important for managing sepsis due to the high dissimilarity among sepsis patients \citep{lat2021,Komorowski,SonabendESRL}. Physicians usually treat it with a high, constant dose of antibiotics. They also use vasopressors to control blood pressure. Not all patients benefit from vasopressors, however, and it is also unclear when to administer it \citep{lat2021}. Here $d_t$ corresponds to the decision regarding whether  vasopressors should be administered at a given time $t$, where we consider $t=1$ at baseline and $t=2$ at $4$ hour after diagnosis of sepsis. We code our actions as $A_t=-1$ if no dose is necessary and $A_t=1$ otherwise. The state space is comprised of 46 covariates $O_t$ at each time step. Measured variables include age, body mass index, diastolic and systolic blood pressure, etc. We use an inverse transformation of lactate acid level as the outcome $Y_t$. In particular $Y_t=(LA_t+5)^{-1}+2$ where $LA_t$ stands for lactic acid level at time $t$. Adding the offset $2$ ensures that $Y_t>0$, which is required by Assumption \ref{assump: bound on Y}. 

We evaluate DTRESLO using linear, wavelets, splines and neural network functions. The comparators are as in Section \ref{sec: simulation}. For DTRESLO, we   choose from four choices of $\phi$  presented in Example~\ref{Example: sigmoid}.  To this end, we use a five-fold cross-validation. 
To elaborate, for each split, 
 we estimate the DTRs using 4 folds of the sample. Then we estimate the value function of each treatment regime using the remaining fold.
Thus  DTRESLO with above-mentioned four choices of $\phi$ yield four potentially different treatment regimes in the first step. However, we pick only one regime among these four regimes. We consider that treatment regime, which has the highest value function estimate. 
 We also report the mean and SE of the value function estimates for both DTRESLO and the comparators.

We estimate the value functions of the derived DTRs using a doubly robust approach \citep{Jiang2016,WDR,sonabendw2021semisupervised}. The corresponding value function estimator is given by
\begin{align*}
\begin{split}
\widehat V_{DR}(\widehat d_1,\widehat d_2)=&
\mathbb P_n\bigg[
\widehat Q_1(H_1,\hat d_1(H_1))\\
+&\frac{1[A_1=\hat d_1(H_1)]}{\widehat\pi_1(A_1|H_1)}
\left[
Y_1-\left\{\widehat  Q_1(H_1,\hat d_1(H_1))- \widehat  Q_2(H_2,\hat d_2(H_2))
\right\}\right]\\
+&\frac{1[A_1=\hat d_1(H_1),A_2=\hat d_2(H_2)]}{\widehat\pi_1(A_1|H_1)\widehat\pi_2(A_2|H_2)}\left\{
Y_2-\widehat  Q_2(H_2,\hat d_2(H_2))
\right\}\bigg],
\end{split}
\end{align*}
{ where $\widehat \pi_1$ and $\widehat \pi_2$ are the propensity score estimators, and  $\widehat Q_1$ and $\widehat Q_2$ are  estimators of the Q-functions corresponding to the first and second stage, respectively. Here the $Q$ functions are estimated using a neural network  with the same specifications as our nonparametric Q-learning DTR estimator.}

{
}

\setlength{\tabcolsep}{3pt}
\begin{table}[H]
\resizebox{\columnwidth}{!}{
\begin{tabular}{ccccc|cc|cc|cc}
\hline
                            & \begin{tabular}[c]{@{}c@{}}DTRESLO\\ (linear)\end{tabular} & \begin{tabular}[c]{@{}c@{}}DTRESLO\\ (splines)\end{tabular} &
                            \begin{tabular}[c]{@{}c@{}}DTRESLO\\ (wavelets)\end{tabular} &
                            \begin{tabular}[c]{@{}c@{}}DTRESLO\\ (NN)\end{tabular} & \begin{tabular}[c]{@{}c@{}}BOWL\\ (linear)\end{tabular}           & \begin{tabular}[c]{@{}c@{}}BOWL\\ (RBF)\end{tabular}           & \begin{tabular}[c]{@{}c@{}}SOWL\\ (linear)\end{tabular}           & \begin{tabular}[c]{@{}c@{}}SOWL\\ (RBF)\end{tabular}           & \begin{tabular}[c]{@{}c@{}}Q-learning\\ (linear)\end{tabular} & \begin{tabular}[c]{@{}c@{}}Q-learning\\ (NN)\end{tabular} \\ \cline{2-11} 
$\hat V_{DR}(\hat d_1,\hat d_2)$ & 1.352          &   1.415                                               &    \textbf{1.467}                                                    &    1.458                                                      &1.248 &1.450 &  0.845 & 1.137          & 1.062                                                         & 1.144                                                     \\
SD                     &    0.043 & 0.084                                                &    0.024                               &    0.023                                                & 0.039 & 0.071& 0.148& 0.082 & 0.095                           & 0.102                               \\ \hline
\end{tabular}}\caption{Estimated Value function and Standard deviation for the DTRs derived with different methods. Here  NN stands for neural network.}\label{Table: sepsis values}
\end{table}

Table \ref{Table: sepsis values} reports the estimated values using the doubly robust estimator for each method and their corresponding standard errors. 
We observe that DTRESLO with wavelet treatment policies has the highest value function estimate. Also, the estimated standard deviations of non-linear (nonparametric) DTRESLO estimators are smaller compared to the nonparametric counterparts of BOWL, SOWL, and Q-learning.
It must be noted that the linear DTRESLO outperforms linear $Q$-learning, which was  observed in all our simulation settings as well.
Also, as in our simulations,
SOWL has lower value function estimate compared to all other methods. This aligns with our observations with SOWL in Section~\ref{sec: simulation}.


  \section{Discussion}
  \label{sec: discussion}
 {\color{black}  Our work\label{page: first step} \emph{is the first step} towards a unified understanding of general surrogate losses in the simultaneous optimization  context.} Our work leaves ample room for  modification and generalization to complex real-world scenarios. We list below some important open questions:

{\color{black} \label{page: open problem} 
Regarding the optimization error, we have analyzed linear-type treatment policies under conditions with a primary focus on  landscape analysis.  However, our simulations in Section \ref{sec: simulation} indicates that DTRESLO performs competitively to popular  DTR methods, regardless of whether the policies are linear or non-linear. Therefore,  a more comprehensive analysis of the optimization error is required to gain deeper insight into the performance of   DTRESLO.}

  The theoretical results in this paper consider the propensity scores to be known. They may be available in SMART studies, but they need to be estimated for observational studies. At best, we may be able to estimate the propensity scores at $n^{-1/2}$-rate. Therefore it is possible that in this situation,  our regret-decay rate will slow down. We also do not know if it is at all possible to push the regret decay to $O(1/n)$ in this situation because we do not know the minimax rate of regret-decay in this context. However, there is a more pressing issue with the use of inverse propensity score weighting. The weight will grow smaller as the number of stages increases, leading to a highly  volatile method \citep{laber2019}. However, there are strategies \citep{kallus2018} that can be incorporated to ensure robustness. Research in this direction is needed to increase the stability of our DTRESLO method. 
  
  Also, there are many choices of  $\phi$'s that satisfy Condition~\ref{cond: sigmoid  condition}, and hence can be used for DTRESLO.  In this paper, we have not considered the problem of selecting a $\phi$. We fixed a particular $\phi$ in our empirical study but the performance may be improved by a more careful tuning of $\phi$. 
  
  The DTRESLO method easily extends to $K > 2$ by using a surrogate $\psi(x_1,\ldots,x_K)=\phi(x_1)\ldots\phi(x_K)$. Although we do not yet know whether  Fisher consistency still holds, our proof techniques are readily extendable to the higher stages via mathematical induction. If our DTRESLO method is Fisher consistent for general $K$ stages,  the pattern of error accumulation over stages will be an immediate interest. For Q-learning, the regret grows exponentially with the number of stages \citep{murphy2005}.  In view of \cite{wang2020statistical},  exponential error accumulation may sometimes be inevitable under very  general  conditions. However,  we wonder whether our simultaneous maximization procedure escapes the exponential error accumulation in the presence of noise conditions.

  Despite being of immense practical interest, this area greatly lacks direct search method with rigorous guarantees in multi-stage settings. Direct search method with more than two levels of treatment requires integration of multicategory classification with the sequential setting of DTR, and hence is conceptually more challenging than the regression-based counterparts. However, we expect that DTRESLO can be extended to identify optimal DTRs under this more complex setting. 
  Detailed strategies for identifying the surrogate loss and implementing algorithms to estimate DTRs in practice warrant future research.

 \section{Acknowledgments}
Rajarshi Mukherjee and Nilanjana Laha's research was supported by National Institutes of Health grant P42ES030990.  Tianxi Cai's research was supported by National Institutes of Health grant R01LM013614. Aaron Sonabend's research was supported by the Boehringer-Ingelheim Fellowship at Harvard. 

  \section{Supplement}
  Due to the size of the Supplement, it has been uploaded  on the first  author's website, and can be accessed using this  \href{https://sites.google.com/d/1KLrABaa49Yb-z0SjeUPxUruj69ktBCEQ/p/1FuwCoWWhBXb2DfKg5E0UiFOu8xYVHxKJ/edit}{link}.
  The supplement includes discussions on optimization error, additional details about the hinge loss, further exploration of various assumptions made in this paper, additional details regarding the simulation settings in Section \ref{sec: simulation}, and the proofs.

 \bibliographystyle{natbib}
\bibliography{DTR,DTR2}
\end{document}